\documentclass[a4paper, british]{amsart}

\usepackage{amssymb}
\usepackage{babel}
\usepackage{enumitem}
\usepackage{hyperref}
\usepackage[utf8]{inputenc}
\usepackage{newunicodechar}
\usepackage{varioref}
\usepackage[arrow,curve,matrix]{xy}

\usepackage{colortbl}
\usepackage{graphicx}
\usepackage{tikz}

\usepackage{mathrsfs}

\definecolor{linkred}{rgb}{0.7,0.2,0.2}
\definecolor{linkblue}{rgb}{0,0.2,0.6}

\numberwithin{figure}{section}

\usepackage[hyperpageref]{backref}

\setdescription{labelindent=\parindent, leftmargin=2\parindent}
\setitemize[1]{labelindent=\parindent, leftmargin=2\parindent}
\setenumerate[1]{labelindent=0cm, leftmargin=*, widest=iiii}

\newunicodechar{α}{\ensuremath{\alpha}}
\newunicodechar{β}{\ensuremath{\beta}}
\newunicodechar{χ}{\ensuremath{\chi}}
\newunicodechar{δ}{\ensuremath{\delta}}
\newunicodechar{ε}{\ensuremath{\varepsilon}}
\newunicodechar{Δ}{\ensuremath{\Delta}}
\newunicodechar{η}{\ensuremath{\eta}}
\newunicodechar{γ}{\ensuremath{\gamma}}
\newunicodechar{Γ}{\ensuremath{\Gamma}}
\newunicodechar{ι}{\ensuremath{\iota}}
\newunicodechar{κ}{\ensuremath{\kappa}}
\newunicodechar{λ}{\ensuremath{\lambda}}
\newunicodechar{Λ}{\ensuremath{\Lambda}}
\newunicodechar{μ}{\ensuremath{\mu}}
\newunicodechar{ω}{\ensuremath{\omega}}
\newunicodechar{Ω}{\ensuremath{\Omega}}
\newunicodechar{π}{\ensuremath{\pi}}
\newunicodechar{φ}{\ensuremath{\phi}}
\newunicodechar{Φ}{\ensuremath{\Phi}}
\newunicodechar{ψ}{\ensuremath{\psi}}
\newunicodechar{Ψ}{\ensuremath{\Psi}}
\newunicodechar{ρ}{\ensuremath{\rho}}
\newunicodechar{σ}{\ensuremath{\sigma}}
\newunicodechar{Σ}{\ensuremath{\Sigma}}
\newunicodechar{τ}{\ensuremath{\tau}}
\newunicodechar{θ}{\ensuremath{\theta}}
\newunicodechar{Θ}{\ensuremath{\Theta}}

\newunicodechar{𝔹}{\ensuremath{\bB}}
\newunicodechar{ℂ}{\ensuremath{\bC}}
\newunicodechar{ℕ}{\ensuremath{\bN}}
\newunicodechar{ℙ}{\ensuremath{\bP}}
\newunicodechar{ℚ}{\ensuremath{\bQ}}
\newunicodechar{ℝ}{\ensuremath{\bR}}
\newunicodechar{ℤ}{\ensuremath{\bZ}}

\newunicodechar{∇}{\ensuremath{\nabla}}

\newunicodechar{⊕}{\ensuremath{\oplus}}
\newunicodechar{⊗}{\ensuremath{\otimes}}
\newunicodechar{Λ}{\ensuremath{\wedge}}
\newunicodechar{→}{\ensuremath{\to}}
\newunicodechar{⨯}{\ensuremath{\times}}
\newunicodechar{∪}{\ensuremath{\cup}}
\newunicodechar{∩}{\ensuremath{\cap}}
\newunicodechar{⊇}{\ensuremath{\supseteq}}
\newunicodechar{⊃}{\ensuremath{\supset}}
\newunicodechar{⊆}{\ensuremath{\subseteq}}
\newunicodechar{⊂}{\ensuremath{\subset}}
\newunicodechar{≥}{\ensuremath{\geq}}
\newunicodechar{≤}{\ensuremath{\leq}}
\newunicodechar{∈}{\ensuremath{\in}}
\newunicodechar{◦}{\ensuremath{\circ}}
\newunicodechar{°}{\ensuremath{^\circ}}
\newunicodechar{…}{\ifmmode\mathellipsis\else\textellipsis\fi}
\newunicodechar{·}{\ensuremath{\cdot}}
\newunicodechar{∅}{\ensuremath{\emptyset}}

\newunicodechar{¹}{\ensuremath{^1}}
\newunicodechar{²}{\ensuremath{^2}}
\newunicodechar{³}{\ensuremath{^3}}

\DeclareFontFamily{OMS}{rsfs}{\skewchar\font'60}
\DeclareFontShape{OMS}{rsfs}{m}{n}{<-5>rsfs5 <5-7>rsfs7 <7->rsfs10 }{}
\DeclareSymbolFont{rsfs}{OMS}{rsfs}{m}{n}
\DeclareSymbolFontAlphabet{\scr}{rsfs}
\DeclareSymbolFontAlphabet{\scr}{rsfs}

\DeclareMathOperator{\Aut}{Aut}
\DeclareMathOperator{\codim}{codim}

\DeclareMathOperator{\Id}{Id}

\DeclareMathOperator{\Pic}{Pic}
\DeclareMathOperator{\rank}{rank}

\DeclareMathOperator{\reg}{reg}

\DeclareMathOperator{\sEnd}{\mathcal{E}\negthinspace \mathit{nd}}
\DeclareMathOperator{\sing}{sing}

\DeclareMathOperator{\supp}{supp}

\newcommand{\sE}{\scr{E}}
\newcommand{\sF}{\scr{F}}
\newcommand{\sG}{\scr{G}}

\newcommand{\sL}{\scr{L}}

\newcommand{\sO}{\scr{O}}

\newcommand{\sT}{\scr{T}}

\newcommand{\cA}{\mathcal A}
\newcommand{\cC}{\mathcal C}
\newcommand{\cD}{\mathcal D}
\newcommand{\cV}{\mathcal V}

\newcommand{\bB}{\mathbb{B}}
\newcommand{\bC}{\mathbb{C}}
\newcommand{\bD}{\mathbb{D}}

\newcommand{\bN}{\mathbb{N}}

\newcommand{\bP}{\mathbb{P}}
\newcommand{\bQ}{\mathbb{Q}}
\newcommand{\bR}{\mathbb{R}}

\newcommand{\bZ}{\mathbb{Z}}

\theoremstyle{plain}   
\newtheorem{thm}{Theorem}[section]

\newtheorem{cor}[thm]{Corollary}
\newtheorem{defn}[thm]{Definition} 

\newtheorem{lem}[thm]{Lemma}

\newtheorem{prop}[thm]{Proposition}

\theoremstyle{remark}
\newtheorem{assumption}[thm]{Assumption} 

\newtheorem{claim}[thm]{Claim}
\newtheorem{c-n-d}[thm]{Claim and Definition}

\newtheorem{construction}[thm]{Construction}

\newtheorem{example}[thm]{Example}

\newtheorem{notation}[thm]{Notation}

\newtheorem{rem}[thm]{Remark}

\newtheorem*{rem-nonumber}{Remark}

\newtheorem{warning}[thm]{Warning}

\numberwithin{equation}{thm}

\setlist[enumerate]{label=(\thethm.\arabic*), before={\setcounter{enumi}{\value{equation}}}, after={\setcounter{equation}{\value{enumi}}}}

\newcommand{\into}{\hookrightarrow}

\newcommand{\wtilde}{\widetilde}
\newcommand{\what}{\widehat}

\def\clap#1{\hbox to 0pt{\hss#1\hss}}

\def\mathclap{\mathpalette\mathclapinternal}

\def\mathclapinternal#1#2{%
\clap{$\mathsurround=0pt#1{#2}$}}

\newcommand\CounterStep{\addtocounter{thm}{1}\setcounter{equation}{0}}

\newcommand{\factor}[2]{\left. \raise 2pt\hbox{$#1$} \right/\hskip -2pt\raise -2pt\hbox{$#2$}}

\makeatletter
\let\saveqed\qed
\renewcommand\qed{%
   \ifmmode\displaymath@qed
   \else\saveqed
   \fi}
\makeatother
\newcommand{\Publication}[1]{}

\newcommand{\subversionInfo}{}
\newcommand{\svnid}[1]{}
\newcommand{\approvals}[1]{}

\usepackage{mathpazo}

\author{Daniel Greb}
\address{Daniel Greb, Essener Seminar für Algebraische Geometrie und Arithmetik, Fakultät für Mathematik, Universität Duisburg--Essen, 45117 Essen, Germany}
\email{\href{mailto:daniel.greb@uni-due.de}{daniel.greb@uni-due.de}}
\urladdr{\href{http://www.esaga.uni-due.de/daniel.greb}{http://www.esaga.uni-due.de/daniel.greb}}

\author{Stefan Kebekus}
\address{Stefan Kebekus, Mathematisches Institut, Albert-Ludwigs-Universität Freiburg, Eckerstraße 1, 79104 Freiburg im Breisgau, Germany}
\email{\href{mailto:stefan.kebekus@math.uni-freiburg.de}{stefan.kebekus@math.uni-freiburg.de}}
\urladdr{\href{http://home.mathematik.uni-freiburg.de/kebekus}{http://home.mathematik.uni-freiburg.de/kebekus}}

\thanks{Stefan Kebekus gratefully acknowledges the support through a joint
  fellowship of the Freiburg Institute of Advanced Studies (FRIAS) and the
  University of Strasbourg Institute for Advanced Study (USIAS). Daniel Greb was partially supported by the DFG-Collaborative Research
  Center SFB/TR 45 ``Periods, Moduli and Arithmetic of Algebraic Varieties''. Behrouz Taji
  was partially supported by the DFG-Graduiertenkolleg GK1821 ``Cohomological
  Methods in Geometry'' at Freiburg.}

\author{Behrouz Taji}
\address{Behrouz Taji, Mathematisches Institut, Albert-Ludwigs-Universität Freiburg, Eckerstraße 1, 79104 Freiburg im Breisgau, Germany}
\email{\href{mailto:behrouz.taji@math.uni-freiburg.de}{behrouz.taji@math.uni-freiburg.de}}
\urladdr{\href{http://home.mathematik.uni-freiburg.de/taji}{http://home.mathematik.uni-freiburg.de/taji}}



\title{Uniformisation of higher-dimensional minimal varieties}
\date{\today}

\makeatletter
\hypersetup{
  pdfauthor={\authors},
  pdftitle={\@title},
  pdfsubject={\@subjclass},
  pdfkeywords={\@keywords},
  pdfstartview={Fit},
  pdfpagelayout={TwoColumnRight},
  pdfpagemode={UseOutlines},
  bookmarks,
  colorlinks,
  linkcolor=linkblue,
  citecolor=linkred,
  urlcolor=linkred}
\makeatother

\DeclareMathOperator{\Div}{Div}

\DeclareMathOperator{\GL}{GL}

\newcommand{\drefl}{d_{\operatorname{refl}}}

\newcommand{\pCVHS}{{\sf p$ℂ$VHS}{}}

\theoremstyle{remark}

\newtheorem{reminder}[thm]{Reminder}
\newtheorem*{goal}{Goal}

\begin{document}

\maketitle
\tableofcontents

%
%
\svnid{$Id: 01.tex 125 2016-08-19 06:43:23Z kebekus $}

\section{Introduction}
\subversionInfo

\subsection{Uniformisation of Riemann surfaces}
\approvals{Behrouz & yes \\Daniel & yes \\Stefan & yes}

One of the main reasons for the beauty and elegance of Riemann surface theory is
the fact that there is a very short list of simply-connected Riemann surfaces,
given by the uniformisation theorem of Koebe and Poincaré.

\begin{thm}[Uniformisation theorem for Riemann surfaces, 1907]
  Let $X$ be a simply-connected Riemann surface.  Then, $X$ is biholomorphic to
  exactly one of the following Riemann surfaces:
  \begin{itemize}
  \item the Riemann sphere $\what{ℂ} = ℂ ∪ \{\infty\}$,
  \item the complex plane $ℂ$,
  \item the unit disk $\bD = \{z ∈ ℂ \mid |z| < 1\}$.
  \end{itemize}
\end{thm}

The history of this result is rich, begins with the Riemann mapping theorem in
the 1850s, and involves many of the most important figures of mathematics at the
end of the $19$th and the beginning of the $20$th century.  It is surveyed for
example in \cite{MR1295591}.

One of the main consequences of the uniformisation theorem is the ``great
trichotomy'' seen in the geometry of \emph{compact} Riemann surfaces,
cf.~\cite{MR1957061}: they fall into three classes that can
be defined and characterised in topological\footnote{Euler characteristic
  positive, zero or negative; genus $0$, $1$ or $≥ 2$; fundamental group
  trivial, Abelian or nonabelian}, numerical\footnote{degree of canonical line
  bundle negative, zero or positive}, algebro-geometric\footnote{canonical line
  bundle anti-ample, trivial or ample}, and
differential-geometric\footnote{existence of a metric of constant positive, zero
  or positive curvature} terms.

Some of the equivalences contained in the previous enumeration can actually be
proven using the uniformisation theorem: The existence of constant curvature
metrics can be deduced from the existence of constant curvature metrics on the
universal covers that are invariant under the respective deck transformation
groups: the Fubini-Study metric on $\what{ℂ} = ℙ^1$, the standard flat metric on
$ℂ$, and the Poincaré metric on $\bD$.

One further aspect that will appear again later is the fact that while the
theorem of Gauß-Bonnet tells us that the integral over the curvature of
\emph{any} Riemannian metric on a compact Riemann surface $X$ equals $2-2g(X)$,
there always exists a \emph{distinguished} metric of constant curvature, whose
sign is dictated by the topology.  Moreover, we can determine the universal
cover of a given compact Riemann surface up to biholomorphism just by computing
the degree of the canonical bundle on the surface itself.

In some sense, the uniformisation theorem reduces the study of the geometry of
(compact) Riemann surfaces to the investigation of the $π_1$-equivariant
geometry of the universal cover.  The corresponding function theory is the study
of theta functions in the genus zero case, and of modular forms in the higher
genus case.  One result that can be obtained using the study of these special
functions is that every compact Riemann surface is in fact projective, see for
example \cite[Chapter~IX]{MR1328834}.

\subsection{Problems in higher dimensions}
\approvals{Behrouz & yes \\Daniel & yes \\Stefan & yes}

Moving on to higher dimensions, one quickly realises that a statement similar to
Koebe's and Poincaré's result is not possible, as many new phenomena appear.  In
some sense, there are just too many simply-connected complex manifolds in any
given dimension, as exemplified by the following.

\begin{itemize}
\item[a)] The only simply-connected compact Riemann surface is $ℙ^1$, defining
  the class of the trichotomy with negative canonical bundle.  On the other
  hand, Lefschetz' theorem implies that for example any smooth hypersurface of
  degree greater than or equal to five in $ℙ^3$ is simply-connected.  The
  canonical bundle of such a surface is ample and it therefore belongs to the
  opposite end of the spectrum.  At the same time, this yields non-trivial
  families of (compact) simply-connected non-biholomorphic manifolds, which also
  do not exist in dimension one.

\item[b)] Every complex manifold with ample anti-canonical bundle, i.e., every
  \emph{Fano manifold}, is rationally chain-connected by work of Campana
  \cite{Campana92} and Kollár-Miyaoka-Mori \cite{KMM92}, and hence
  simply-connected, see for example \cite[Thm.~3.5]{Campana91}.  Moreover, Fano
  manifolds of fixed dimension form a bounded family, see again \cite{KMM92}.
  While classification up to deformation was achieved in dimension three by
  Iskovskih \cite{MR463151, MR503430} and Mori-Mukai \cite{MR641971, MR1969009}
  at the beginning of the 1980's building on work of Shokurov \cite{MR534602},
  it is impossible in higher dimensions.

\item[c)] \emph{Calabi-Eckmann manifolds} are complex manifolds whose underlying
  real-differentiable manifold is isomorphic to the product $S^{2n+1}⨯ S^{2m+1}$
  of to odd-dimensional spheres, see \cite{MR0057539}.  They contain open
  subsets diffeomorphic to $ℝ^{2(n+m)+2}$ which do not admit any non-constant
  holomorphic function.  In particular, the corresponding complex structure on
  $ℝ^{2(n+m)+2}$ cannot be described using a single coordinate chart.

\item[d)] On the other hand, looking at basins of attractions for (the iteration
  of) certain holomorphic maps $f: ℂ^n → ℂ^n$, one finds open subsets of $ℂ^n$
  that are biholomorphic to $ℂ^n$, the so-called \emph{Fatou-Bieberbach
    domains}; see \cite[Chapter~6.3]{MR1747010}.  These examples stand in sharp
  contrast to the statement of the Riemann mapping theorem in dimension one.

\item[e)] The unit disk is the only Riemann surface in the list that admits
  non-constant bounded holomorphic functions.  In higher dimensions, there are
  many more examples: First shown by Poincaré, the unit ball
  $$
  𝔹^n = \bigl\{(z_1, \dots, z_n) ∈ ℂ^n \mid z_1^2 + \cdots + z_n^2 < 1 \bigr\}
  $$
  and the unit polydisk $\bD⨯ \cdots ⨯ \bD$ are not biholomorphically equivalent
  for $n≥ 2$, see for example \cite[Chapter~5, Proposition~4]{MR1324108}.  In
  fact, every bounded homogeneous domain is contractible by \cite{MR0158415} and
  hence simply-connected.  Bounded symmetric domains were classified by Cartan,
  see for example \cite[Chapter~X]{MR1834454}, but starting in dimension four
  not every bounded homogeneous domain is symmetric, as shown by a famous
  example of Piatetski-Shapiro, \cite{MR0252690}.  Starting in dimension $7$
  there are infinite families of bounded homogeneous domains that are not
  symmetric.

\item[f)] Given one simply-connected complex manifold $X$ of dimension greater
  than one, one can produce infinitely many new ones by blowing-up points in
  $X$.  Hence, to achieve some understanding one certainly has to impose some
  minimality condition.  This point will reappear in Section~\ref{subsect:MMP}.
\end{itemize}

One solution to the issue raised by the enumeration above is to modify the
question and ask for a characterisation of those compact complex
manifolds/smooth projective varieties whose universal cover is biholomorphic to
some fixed simply-connected model space having well-understood geometry.
Motivated by the uniformisation theorem for Riemann surfaces, in the following
discussion we hence concentrate on the following

\begin{goal}
  Characterise those compact complex manifolds whose universal cover is
  biholomorphic to $ℙ^n$, $ℂ^n$, or $𝔹^n$.
\end{goal}

In fact, as in dimension one, every holomorphic automorphism of $ℙ^n$ has a
fixed point, and so for this part one is left with the task of characterising
$ℙ^n$.  The techniques used in various approaches to this problem are mostly
based on studying rational curves and are hence different in spirit from the
other two cases.  We refer the reader to \cite{Mori79, MR577360, Wah83, AW01,
  CMSB02, Kebekus02b, MR2439607} and from
Section~\ref{subsect:deriving_conditions} onwards restrict ourselves to studying
quotients of $ℂ^n$ and $𝔹^n$.  We note that as in the case of Riemann surfaces,
the geometry of these manifolds can be studied using $π_1$-equivariant objects
on the universal cover, theta functions and automorphic forms.

\subsection{Metric characterisations}
\approvals{Behrouz & yes \\Daniel & yes \\Stefan & yes}

From the differential-geometric formulation of the great trichotomy, one derives
the idea that searching for special metrics is one approach to the
uniformisation problem also in higher dimensions.  And indeed, one has the
following result; cf.~\cite[Chapter~IX, Theorem.~7.9]{MR1393941}, where it is
credited to Hawley and Igusa:

\begin{thm}\label{thm:constantcurvature}
  If a projective manifold admits a Kähler metric of constant holomorphic
  sectional curvature, the universal cover of $X$ is biholomorphic to $ℙ^n$,
  $ℂ^n$, or $𝔹^n$ (depending on the sign of the curvature).
\end{thm}

However, determining whether a given projective manifold admits a Kähler metric
of constant curvature is a difficult task, and \emph{a priori} not an
algebro-geometric (or even topological) condition.

\subsection{Deriving necessary algebro-geometric conditions}
\label{subsect:deriving_conditions}
\approvals{Behrouz & yes \\Daniel & yes \\Stefan & yes}

Let $X$ be a projective manifold whose universal cover is biholomorphic to
$ℂ^n$.  It was conjectured by Iitaka and proven by Nakayama in dimension less
than or equal to three (and in all dimensions assuming the Abundance Conjecture)
that in this situation there exists an Abelian variety $A$ and a finite group
$G$ of fixed-point free holomorphic automorphisms of $A$ such that
$X \cong A/G$, \cite[Theorem~1.4]{MR1691481}.  Consequently, the tangent bundle
of $X$ is flat, and therefore we obtain the intersection-theoretic conditions
$c_1(X)=0 ∈ H²\bigl(X,\, ℝ\bigr)$ and $c_2(X)·[H]^{n-2}=0$, where $H$ is any
ample divisor on $X$.  In particular, we have
\begin{equation}\label{eq:MY=1}
  \Bigl(c_2(X)- \frac{n}{2(n+1)}·c_1^2(X) \Bigr)·[H]^{n-2} = 0.
\end{equation}

If $X$ is a projective manifold whose universal cover is biholomorphic to $𝔹^n$,
then the Bergman metric of $𝔹^n$, which has constant negative holomorphic
sectional curvature, is invariant under the deck transformation group.  It
induces a Kähler metric $g$ on $X$ whose associated $(1,1)$-form is the
curvature form of a metric in the canonical bundle of $X$, which is therefore
ample by Kodaira's theorem.  Note that ampleness can be detected using
intersection theory by the Nakai-Moishezon criterion.  Additionally, using the
fact that $g$ has constant holomorphic sectional curvature and that we can
compute the Chern classes of $X$ from $g$, or using the Hirzebruch
proportionality principle \cite[Appendix~1]{MR1335917} one concludes that
\begin{equation}\label{eq:MY=2}
  \Bigl( c_2(X)- \frac{n}{2(n+1)}·c_1^2(X) \Bigr)·[K_X]^{n-2} = 0.
\end{equation}

\subsection{The Miyaoka-Yau inequality and uniformisation for higher-dimensional manifolds}
\label{subsect:MYIsmooth}
\approvals{Behrouz & yes \\Daniel & yes \\Stefan & yes}

In fact, the two equations \eqref{eq:MY=1} and \eqref{eq:MY=2} represent the
extremal case of an inequality and they characterise exactly those projective
manifolds whose universal cover is isomorphic to $ℂ^n$ and $𝔹^n$ by the
following result of Yau \cite{MR0451180}.

\begin{thm}[Miyaoka-Yau inequality]\label{thm:YausTheorem}
  Let $X$ be an $n$-dimensional projective manifold whose canonical class is
  numerically trivial or ample, respectively.  Then, we have
 \begin{equation}\label{eq:BMYI}
   \Bigl(c_2(X)- \frac{n}{2(n+1)}·c_1^2(X)\Bigr)·[A]^{n-2} ≥ 0,
 \end{equation}
 where $A$ is either an arbitrary ample divisor on $X$ or equal to $K_X$,
 respectively.  We call \eqref{eq:BMYI} the \emph{Miyaoka-Yau inequality}.  In
 case of equality, the universal cover of $X$ is biholomorphic to $ℂ^n$ or
 $𝔹^n$, respectively.
\end{thm}
\begin{proof}[Sketch of proof]
  The proof is based on Yau's solution of the Calabi conjecture, which in the
  situation at hand produces a Kähler-Einstein metric on $X$ whose associated
  $(1,1)$-form represents the class of $A$ in $H^{1,1}\bigl(X,\, ℝ\bigr)$.
  Using this metric to compute differential forms representing the Chern classes
  of the tangent bundle and exploiting the symmetries of the curvature tensor
  imposed by the Kähler-Einstein condition, one sees that the desired equality
  holds \emph{pointwise} for the chosen differential forms.  The inequality
  \eqref{eq:BMYI} follows by integration.  Yau credits this part of the argument
  to Guggenheimer.

  In case of equality, the pointwise computations done before yield enough
  restrictions on the metric to see that $X$ has constant holomorphic sectional
  curvature; the complete computation can for example be found in
  \cite[pp.~225f]{MR1777835}.  We conclude using
  Theorem~\ref{thm:constantcurvature}.
\end{proof}

Hence, after Yau's result, the question of the existence of a constant curvature
metric in a sense is intersection-theoretic, as it is \emph{a posteriori}
guaranteed by numerical triviality/ampleness of the canonical bundle and
equality in Theorem~\ref{thm:YausTheorem}.  The result is very close in spirit
to the one-dimensional case: again, one can \emph{a priori} use any metric on
$X$ to check whether the Chern classes of $X$ satisfy equality in
\eqref{eq:BMYI}.  However, if this is the case, there exists a distinguished
metric having constant curvature, whose sign again depends on the sign of the
canonical class.

Generalisations of the Miyaoka-Yau inequality, the question whether there is an
algebro-geometric proof, and the problem of uniformisation in case of equality
have attracted considerable interest in the last few decades, see
Section~\ref{subsect:earlierwork} for a discussion.  Here, we only mention that
one important approach to the problem that avoids the construction of
Kähler-Einstein metrics is based on results of Donaldson \cite{Donaldson85},
Uhlenbeck-Yau \cite{UhlenbeckYau86}, and Simpson \cite{MR944577} concerning the
existence of Hermitian Yang-Mills connections in stable holomorphic (Higgs)
bundles.  These metrics, although \emph{a priori} less directly connected to the
geometry of $X$, are then used to conclude that in case equality is attained in
\eqref{eq:BMYI} the tangent bundle is flat (in the numerically trivial case) or
that $\sT_X ⊕ \sO_X$ is projectively flat (in the case of ample canonical
bundle).  We will see later that this second approach generalises to the
singular setup in a natural way.

\subsection{Relation to the minimal model program}
\label{subsect:MMP}\approvals{Behrouz & yes \\Daniel & yes \\Stefan & yes}

In a sense, Theorem~\ref{thm:YausTheorem} gives a satisfactory answer to the
uniformisation question for projective manifolds in higher dimensions.  As it
can be applied to projective manifolds with numerically trivial or ample
canonical bundle, it is natural to look for a way of producing such varieties.
At this point the minimal model program comes into play.

Let $X$ be a projective $n$-dimensional manifold of Kodaira dimension $n$.  In
general, though the canonical divisor is rather positive, it will not be ample.
However, by \cite{BCHM10}, the variety $X$ admits a minimal model $X_{min}$ with
terminal singularities and nef canonical divisor, which is moreover semiample by
the basepoint-free theorem, \cite[Theorem~3.3]{KM98}.  The corresponding
morphism $\varphi: X_{min} → X_{can}$ maps $X_{min}$ birationally onto the
canonical model $X_{can}$ of $X$, which has canonical singularities and
\emph{ample canonical divisor}.  A variety with at worst terminal singularities
and nef canonical divisor will be called \emph{minimal},
cf.~Reminder~\vref{remi:cm}.

At least conjecturally, the picture is the same in the case of projective
manifolds $X$ of Kodaira dimension zero: we expect $X$ to have a minimal model
$X_{min}$ with terminal singularities and \emph{numerically trivial canonical
  divisor}, which then in fact is torsion, due to a theorem of Kawamata
\cite[Theorem~8.2]{Kawamata85}.

In both cases, the fact that we made the canonical divisor of $X$ to have
definite sign on the minimal/canonical model came at the cost of introducing
terminal/canonical singularities\footnote{In fact, for technical reasons it is
  very often necessary to work in the slightly bigger class of klt
  singularities.}.  As a result, Yau's Theorem cannot be applied to outcomes of
the minimal model program.  While existence of singular Kähler-Einstein
structures on varieties with klt singularities and trivial/ample canonical
bundle has been established in \cite{MR2505296}, the asymptotics of the metric
near the singularities is currently not understood well-enough to argue as in
the proof of Theorem~\ref{thm:YausTheorem} sketched above.

\subsection{The Miyaoka-Yau inequality and uniformisation for higher-dimensional minimal varieties: recent results}
\approvals{Behrouz & yes \\Daniel & yes \\Stefan & yes}

In order to formulate the main result discussed in this note, we start with the
following observations concerning the singularities of minimal and canonical
models that we have to deal with:

If $X$ is a variety with terminal singularities\footnote{for example the minimal
  model of a projective manifold of Kodaira dimension zero}, then the singular
locus of $X$ has codimension at least three.  When this is the case, we say that
$X$ is \emph{smooth in codimension two}.  As a consequence, the localisation
sequence for Chow groups, \cite[Chapter 1, Proposition~1.8]{Fulton98}, allows to
define first and second Chern classes of coherent sheaves, as in the
non-singular situation.  Furthermore, every $2$-dimensional klt singularity is
analytically equivalent to a quotient singularity $ℂ^2/G$, where $G$ is a finite
subgroup of $\GL_2(ℂ)$.  Consequently, for every klt variety $X$ there exists a
closed subvariety $Z$ of codimension at least three such that $X\setminus Z$ has
at worst quotient singularities.  In this case, we say that $X$ \emph{has
  quotient singularities in codimension two}.  This allows to define first and
second \emph{orbifold Chern classes}\footnote{also called ``$ℚ$-Chern classes''}
of reflexive sheaves, written as $\widehat{c}_1(\sE)$, $\widehat{c}_2(\sE)$.  In
particular, rational intersection numbers of $\widehat c_2(\sE)$ with
$(n-2)$-tuples of Cartier divisors exist.

We can now formulate the main results discussed in these notes.  These show that
the fundamental Chern class inequalities continue to hold in the singular
setting, characterise singular torus- and ball-quotients in terms of Chern
classes, and give purely numerical criteria to guarantee that a space with klt
singularities has in fact only quotient singularities.

\begin{thm}[\protect{Characterisation of singular quotients of Abelian varieties, cf.~\cite[Theorem~1.17]{GKP13}}]\label{thm:TQ}
  Let $X$ be a normal, complex, projective variety of dimension $n$ with at
  worst canonical singularities.  Assume that $X$ is smooth in codimension two
  and that the canonical divisor is numerically trivial, $K_X \equiv 0$.
  Further, assume that there exists an ample Cartier divisor $H ∈ \Div(X)$ such
  that $c_2(\sT_X) · [H]^{n-2} = 0$.  Then, there exists an Abelian variety $A$
  and a finite, surjective, Galois morphism $A → X$ that is étale in codimension
  two.
\end{thm}

In other words, once the assumptions of Theorem~\ref{thm:TQ} are fulfilled for
$X$, we can realise it as the quotient of an Abelian variety by a finite group
whose fixed points lie in codimension three or higher.  In particular, in this
case $X$ has at worst quotient singularities.  In a sense, the map $A → X$
provides a \emph{singular uniformisation} of $X$, cf.~Nakayama's result
discussed in Section~\ref{subsect:deriving_conditions}.  Generalisations to klt
spaces have been obtained in \cite{LT14}.  The proof of Theorem~\ref{thm:TQ}
presented in Section~\ref{subsect:torusproof} uses the inequality
$c_2(\sT_Z) · [H]^{n-2} ≥ 0$, proven by Miyaoka \cite{Miyaoka87}, that holds for
any canonical variety $Z$ that is smooth in codimension two and whose canonical
divisor is numerically trivial.

\begin{thm}[\protect{$ℚ$-Miyaoka-Yau inequality, \cite[Theorem~1.1]{GKPT15}}]\label{thm:MYinequality}
  Let $X$ be an $n$-dimensional, projective, klt variety of general type whose
  canonical divisor $K_X$ is nef.  Then,
  \begin{equation}\label{eq:X2}
    \Bigl( 2(n+1)· \widehat{c}_2(\sT_X)-n·\widehat{c}_1(\sT_X)^2
    \Bigr)·[K_X]^{n-2}≥ 0.
  \end{equation}
\end{thm}

\begin{thm}[\protect{Characterisation of singular ball quotients, \cite[Theorem~1.2]{GKPT15}}]\label{thm:BQ}
  Let $X$ be an $n$-dimensional minimal variety of general type.  If equality
  holds in the $ℚ$-Miyaoka-Yau inequality~\eqref{eq:X2}, then the canonical
  model $X_{can}$ is smooth in codimension two, there exists a ball quotient $Y$
  and a finite, Galois, quasi-étale morphism $f: Y → X_{can}$.  In particular,
  $X_{can}$ has only quotient singularities.
\end{thm}

Here, a \emph{ball quotient} is a projective manifold whose universal cover is
the unit ball.  In fact, it can be shown that in the situation of
Theorem~\ref{thm:BQ}, the canonical model $X_{can}$ can be realised as the
quotient of $𝔹^n$ by a properly discontinuous action of $Γ=π_1(X_{can,reg})$
that is free in codimension two, cf.~\cite[Theorem~1.3]{GKPT15}.  The variety
$X_{can}$ in this sense admits a \emph{singular uniformisation} by the unit
ball.  This motivates the term \emph{singular ball quotients}.  We emphasise at
this point that the theory of automorphic forms does not require the discrete
group $Γ$ to act freely on the unit ball, and can therefore be applied to study
the geometry of $X_{can}=𝔹^n/Γ$, see for example \cite[Part II]{Kollar95s}.

Our approach to the proof of the above results is based on stability properties
of (Higgs) sheaves and is motivated by Simpson's approach to the uniformisation
problem alluded to at the end of Section~\ref{subsect:MYIsmooth}.  We generalise
flatness criteria and relevant results of nonabelian Hodge theory to the
singular setting.  In particular, we develop a theory of Higgs sheaves on
singular spaces.  We refer the reader to Section~\ref{subsect:outline} below,
where the contents of this article are described in detail.

\subsection{Earlier work}\label{subsect:earlierwork}
\approvals{Behrouz & yes \\Daniel & yes \\Stefan & yes}

Generalisations of the Miyaoka-Yau inequality and uniformisation in case of
equality have attracted considerable interest in the last few decades.

Inequality~\eqref{eq:BMYI} and the uniformisation result were extended to the
context of compact Kähler varieties with only quotient singularities by
Cheng-Yau~\cite{MR0833802} using orbifold Kähler-Einstein metrics.  Tsuji
established Inequality~\eqref{eq:BMYI} for \emph{smooth minimal models} of
general type in \cite{MR0976585}.  Enoki's result on the semistability of
tangent sheaf of minimal models, \cite{Eno87}, was used by Sugiyama
\cite{MR1145268} to establish the Bogomolov-Gieseker inequality for the tangent
sheaf of any resolution of a given minimal model of general type with only
canonical singularities, the polarisation given by the pullback of the canonical
bundle on the minimal model.  By using a strategy very similar to ours, that is
via results of Simpson~\cite{MR944577}, Langer in~\cite[Thm.~5.2]{MR1954067}
established the Miyaoka-Yau inequality in this context.  He recently also gave
the first purely algebraic proof of the Bogomolov inequality for semistable
Higgs sheaves on smooth projective varieties over fields of arbitrary
characteristic, cf.~\cite{MR3314517}.  A strong uniformisation result, together
with the Miyaoka-Yau inequality, was established by Kobayashi~\cite{MR0799669}
in the case of open orbifold surfaces.

After the work of Tsuji, the past few years have witnessed significant
developments in the theory of singular Kähler-Einstein metrics and Kähler-Ricci
flow.  These are evident, for example, in the works of Tian-Zhang
\cite{Tian-Zhang}, Eyssidieux-Guedj-Zeriahi \cite{MR2505296}, and Zhang
\cite{Zhang06}.  In particular, Inequality \eqref{eq:BMYI} together with
a uniformisation result for \emph{smooth minimal models} of general type have
been successfully established by Zhang~\cite{MR2497488}.

\subsection{Outline of the paper}
\label{subsect:outline}
\approvals{Behrouz & yes \\Daniel & yes \\Stefan & yes}

After introducing some basic notions and definitions in
Sections~\ref{sec:notation} and~\ref{sec:reflDiff}, an important construction is
recalled in Section~\ref{sec:qec}: Maximally quasi-étale covers of mildly
singular spaces over which global, flat, analytic shaves extend across the
singular locus.  Later on, in Sections~\ref{ssec:toriq} and~\ref{sec:bq}, these
covers turn out to be extremely useful for the uniformisation problems.  In
Section~\ref{sec:naht}, Simpson's work on nonabelian Hodge theory is briefly
recalled in a setting that is specifically useful for dealing with the
ball-quotient problem in Section~\ref{sec:bq}.  In Section~\ref{sec:Higgs} we
introduce the notion of Higgs sheaves over singular spaces and briefly discuss
their various fundamental properties.  The material of Sections~\ref{sec:naht}
and \ref{sec:Higgs} is used in Sections~\ref{sec:MY} and \ref{sec:bq}, where we
establish the Miyaoka-Yau inequality and uniformisation by the ball, so the
reader who is only interested in Theorem~\ref{thm:TQ} can safely skip them.

In Section~\ref{ssec:toriq} we work out the sketch of the proof of the
uniformisation by Euclidean space.  Section~\ref{sec:MY} is devoted to
establishing the Miyaoka-Yau inequality.  The main ingredients here are the
stability result of \cite{Guenancia} and the Restriction
Theorem~\ref{thm:restriction}.  The concluding Section~\ref{sec:bq} discusses
the proof of Theorem~\ref{thm:BQ}.

\subsection*{Acknowledgements}
\approvals{Behrouz & yes \\Daniel & yes \\Stefan & yes}

All three authors found the 2015 AMS Summer Research Institute exceptionally
fruitful.  They would like to thank the organisers for the invitation and the
opportunity to present their results.  The authors would also like to thank two
anonymous referees for helpful comments.

This overview article summarises the content of several research articles,
including \cite{GKP11, GKKP11, ExtApplications, GKP13, GKPT15}, which are joint
work with Thomas Peternell.  The results presented here are therefore not new.
The exposition frequently follows the original articles.  There exists some
overlap with \cite{KP14}.


%
%
\svnid{$Id: 02-notation.tex 116 2016-08-18 12:12:37Z kebekus $}

\section{Notation}
\label{sec:notation}

\subsection{Global conventions}
\label{ssec:conventions}
\approvals{Behrouz & yes \\Daniel & yes \\Stefan & yes}

Throughout this paper, all schemes, varieties and morphisms will be defined over
the complex number field.  We follow the notation and conventions of
Hartshorne's book \cite{Ha77}.  In particular, varieties are always assumed to
be irreducible.  For all notation around Mori theory, such as klt spaces and klt
pairs, we refer the reader to \cite{KM98}.

\subsection{Varieties}
\label{ssec:defnVar}
\approvals{Behrouz & yes \\Daniel & yes \\Stefan & yes}

Once in a while, we need to switch between algebraic and analytic categories.
The following notation is then useful.

\begin{notation}[Complex space associated with a variety]
  Given a variety $X$, denote by $X^{an}$ the associated complex space, equipped
  with the Euclidean topology.  If $f : X → Y$ is any morphism of varieties or
  schemes, denote the induced map of complex spaces by
  $f^{an} : X^{an} → Y^{an}$.  If $\sF$ is any coherent sheaf of
  $\sO_X$-modules, denote the associated coherent analytic sheaf of
  $\sO_{X^{an}}$-modules by $\sF^{an}$.
\end{notation}

\begin{defn}[Minimal varieties]\label{def:minimal}
  A normal, projective variety $X$ is called \emph{minimal} if $X$ has at worst
  terminal singularities and if $K_X$ is nef.
\end{defn}

\begin{reminder}[Basepoint-free theorem and canonical models]\label{remi:cm}
  If $X$ is a projective, klt variety of general type whose canonical divisor
  $K_X$ is nef, the basepoint-free theorem asserts that $K_X$ is semiample,
  \cite[Theorem~3.3]{KM98}.  A sufficiently high multiple of $K_X$ thus defines
  a birational morphism $φ: X → Z$ to a normal projective variety with at worst
  klt singularities whose canonical divisor $K_{Z}$ is ample, cf.\
  \cite[Lemma~2.30]{KM98}.  There exists a $ℚ$-linear equivalence
  $K_X \sim_ℚ φ^* K_{Z}$.  If $X$ is a minimal variety of general type, then $Z$
  has at worst canonical singularities.  In this case, we set $Z = X_{can}$, and
  call it the \emph{canonical model} of $X$.
\end{reminder}

\begin{defn}[Ball quotient]\label{defn:BQ}
  A smooth projective variety $X$ of dimension $n$ is a \emph{ball quotient} if
  the universal cover of $X^{an}$ is biholomorphic to the unit ball
  $𝔹^n = \{ (z_1, \dots, z_n) ∈ ℂ^n \mid |z_1|^2 + \dots + |z_n|^2 < 1\}$.
  Equivalently, there exists a discrete subgroup $Γ < \Aut_\sO(𝔹^n)$ of the
  holomorphic automorphism group of $𝔹^n$ such that the action of $Γ$ on $𝔹^n$
  is cocompact and fixed-point free, and such that $X$ is isomorphic to $𝔹^n/Γ$.
\end{defn}

The following will be used for notational convenience.

\begin{notation}[Big and small subsets]
  Let $X$ be a normal, quasi-projective variety.  A closed subset $Z ⊂ X$ is
  called \emph{small} if $\codim_X Z ≥ 2$.  An open subset $U ⊆ X$ is called
  \emph{big} if $X \setminus U$ is small.
\end{notation}

Fundamental groups are basic objects in our arguments.  We will use the
following notation.

\begin{defn}[Fundamental group and étale fundamental group]
  If $X$ is a complex, quasi-projective variety, we set
  $π_1\bigl(X\bigr) := π_1\bigl(X^{an}\bigr)$, and call it the \emph{fundamental
    group of $X$}.  The étale fundamental group of $X$ will be denoted by
  $\what{π}_1\bigl(X\bigr)$.
\end{defn}

\begin{rem}
  Recall that $\what{π}_1(X)$ is isomorphic to the profinite completion of
  $π_1(X)$, cf.~\cite[§5 and references given there]{Milne80}.
\end{rem}

\subsection{Morphisms}
\approvals{Behrouz & yes \\Daniel & yes \\Stefan & yes}

Galois morphisms appear prominently in the literature, but their precise
definition is not consistent.  We will use the following definition, which does
not ask Galois morphisms to be étale.

\begin{defn}[Covers and covering maps, Galois morphisms]\label{def:cover}
  A \emph{cover} or \emph{covering map} is a finite, surjective morphism
  $γ : X → Y$ of normal, quasi-projective varieties.  The covering map $γ$ is
  called \emph{Galois} if there exists a finite group $G ⊂ \Aut(X)$ such that
  $γ$ is isomorphic to the quotient map.
\end{defn}

\begin{defn}[Quasi-étale morphisms]\label{defn:quasietale}
  A morphism $f : X → Y$ between normal varieties is called \emph{quasi-étale}
  if $f$ is of relative dimension zero and étale in codimension one.  In other
  words, $f$ is quasi-étale if $\dim X = \dim Y$ and if there exists a closed,
  subset $Z ⊆ X$ of codimension $\codim_X Z ≥ 2$ such that
  $f|_{X \setminus Z} : X \setminus Z → Y$ is étale.
\end{defn}

\subsection{Intersection and slope}
\approvals{Behrouz & yes \\Daniel & yes \\Stefan & yes}

Given a normal, $n$-dimensional projective variety $X$ and a Cartier divisor
$H ∈ \Div(X)$, we write $[H]$ for its numerical class, ditto with $ℚ$-Cartier
$ℚ$-divisors.  If $H$ is Cartier and $D$ is a Weil-divisor on $X$, there is a
well-defined intersection number between $D$ and $[H]^{n-1}$, which we denote by
$[D]·[H]^{n-1} ∈ ℤ$.  The construction is found in Fulton's book \cite{Fulton98}
and is reviewed in \cite[Section~2.6]{GKPT15}.  In particular, if $\sE$ is any
coherent sheaf, we can associated a Weil divisor to $\det \sE$ and compute its
intersection number with $[H]^{n-1}$.  The result of this operation is written
as $[\sE]·[H]^{n-1} ∈ ℤ$.

\begin{defn}[Slope with respect to a nef divisor]\label{def:slope2}
  Let $X$ be a normal, projective variety and $H$ be a nef $ℚ$-Cartier divisor
  on $X$.  If $\sE \ne 0$ is any torsion free, coherent sheaf on $X$, define the
  \emph{slope of $\sE$ with respect to $H$} as
  $$
  μ_H(\sE) := \frac{ [\sE]·[H]^{\dim X-1}}{\rank \sE}.
  $$
\end{defn}


\part{Techniques}
%
%
\svnid{$Id: 03-ext.tex 116 2016-08-18 12:12:37Z kebekus $}

\section{Reflexive differentials}
\label{sec:reflDiff}
\subversionInfo
\approvals{Behrouz & yes \\Daniel & yes \\Stefan & yes}

Kähler differentials are among the most fundamental objects of algebraic
geometry.  Defined by universal properties, they behave will with respect to
pull-back and form a presheaf on the category of schemes.  Given a singular
space $X$ the sheaves $Ω^p_X$ of Kähler differentials are however generally hard
to deal with.  Even in the simplest of settings, these sheaves have torsion as
well as cotorsion; we refer the reader to the paper \cite{MR2915479} for a
discussion and for a series of elementary examples.

To obtain a more manageable sheaf, we will often consider the double dual of
$Ω^p_X$.  The resulting sheaf of \emph{reflexive differentials} is reflexive,
and thus much better behaved geometrically.  On the downside, reflexive
differentials can not possibly have the universal properties known from Kähler
differentials: since the latter are \emph{defined} by universal properties, any
other construction that satisfies the same universal properties necessarily
gives us the sheaf of Kähler differentials back!  Once we restrict ourselves to
spaces with klt singularities, however, there is more we can say.  It has been
observed in a series of papers by Greb-Kebekus-Kovács \cite{GKK08} and
Greb-Kebekus-Kovács-Peternell \cite{GKKP11} that reflexive differentials do have
some universal properties once we restrict ourselves to (morphisms between) klt
spaces.  This allows to study reflexive differentials in the context of the
minimal model program.  These results have been applied to a variety of
settings, including a study of hyperbolicity of moduli spaces, \cite{KK10}, a
partial generalisation of the Beauville--Bogomolov decomposition theorem
\cite{GKP11}, and deformations of Calabi--Yau varieties \cite{MR3380453}.

\subsection{Definitions and main results}
\approvals{Behrouz & yes \\Daniel & yes \\Stefan & yes}

We briefly recall the relevant definitions and results below.  Since reflexive
differentials have already been discussed in a few other surveys, we restrict
ourselves to the smallest amount of material required in our applications.
There are more general results for dlt and log canonical pairs, including the
existence of residue maps, for which we refer the reader to the references
listed in Section~\ref{sssec:pb-ref} below.

\begin{defn}[Reflexive differentials]
  Given a normal, complex variety $X$, a \emph{reflexive differential} on $X$ is
  a differential form defined only on the smooth locus, without imposing any
  boundary condition near the singularities.  Equivalently, a reflexive
  differential is a section in the double dual of the sheaf of Kähler
  differentials.  Denoting the sheaf of reflexive differentials by $Ω^{[p]}_X$,
  we have
  $$
  Ω^{[p]}_X = \bigl( Ω^p_X \bigr)^{**} = ι_* \bigl( Ω^p_{X_{\reg}} \bigr),
  $$
  where $ι: X_{\reg} → X$ denotes the inclusion of the smooth locus.  More
  generally, given a quasi-projective variety $X$ and a coherent sheaf $\sE$ on
  $X$, write
  $$
  Ω^{[p]}_X := \bigl(Ω^p_X \bigr)^{**}, \quad \sE^{[m]} := \bigl(\sE^{⊗ m}
  \bigr)^{**} \quad\text{and}\quad \det \sE := \bigl( Λ^{\rank \sE} \sE
  \bigr)^{**}.
  $$
  Given any morphism $f : Y → X$, write $f^{[*]} \sE := (f^* \sE)^{**}$, etc.
\end{defn}

The following result asserts the existence of a useful pull-back morphism for
reflexive differentials in the klt setting.

\begin{thm}[\protect{Existence of pull-back morphisms in general, \cite[Theorems~1.3 and 5.2]{MR3084424}}]\label{thm:main}
  Let $f : X → Y$ be any morphism between normal, complex varieties.  Assume
  that there exists a Weil divisor $D$ on $Y$ such that the pair $(Y,D)$ is klt.
  Then there exists a pull-back morphism
  $$
  \drefl f : f^* Ω^{[p]}_Y → Ω^{[p]}_X,
  $$
  uniquely determined by natural universal properties.  \qed
\end{thm}

\subsection{Discussion}
\approvals{Behrouz & yes \\Daniel & yes \\Stefan & yes}

The ``natural universal properties'' mentioned in Theorem~\ref{thm:main} are a
little awkward to formulate.  Precise statements are given in
\cite[Section~5.3]{MR3084424} .  In essence, it is required that the pull-back
morphisms agree with the pull-back of Kähler differentials wherever this makes
sense, and that pull-back is functorial in composition of morphisms.  The
following theorem, which appeared first, is thus a special case, but also forms
a main ingredient in the proof of Theorem~\ref{thm:main}.

\begin{thm}[\protect{Extension theorem, \cite[Theorem~1.4]{GKKP11}}]\label{thm:extpb}
  Let $Y$ be a normal variety and $f: X → Y$ a resolution of singularities.
  Assume that there exists a Weil $ℚ$-divisor $D$ on $Y$ such that the pair
  $(Y,D)$ is klt.  If
  $$
  σ ∈ H^0 \bigl( Y,\, Ω^{[p]}_Y \bigr) = H^0 \bigl( Y_{\reg},\,
  Ω^p_{Y_{\reg}} \bigr)
  $$
  is any reflexive differential form on $Y$, then there exists a differential
  form $τ ∈ H^0 \bigl( X,\, Ω^p_X \bigr)$ that agrees on the complement of the
  $f$-exceptional set with the usual pull-back of the Kähler differential
  $σ|_{Y_{\reg}}$.  \qed
\end{thm}

It should be noted that Theorem~\ref{thm:main} does not require the image of $f$
to intersect the smooth locus of $Y_{\reg}$.  One particularly relevant setting
to which Theorem~\ref{thm:main} applies is that of a klt space $Y$, and the
inclusion (or normalisation) of the singular locus, say $f: X = Y_{\sing} → Y$.
It might seem surprising that a pull-back morphism exists in this context,
because reflexive differential forms on $Y$ are, by definition, differential
forms defined on the \emph{complement} of $Y_{\sing}$, and no boundary
conditions are imposed that would govern the behaviour of those forms near the
singularities.

\subsection{Immediate consequences}
\approvals{Behrouz & yes \\Daniel & yes \\Stefan & yes}

It had been known for a long time that the existence of a pull-back functor for
reflexive forms will give partial answers to the Lipman-Zariski conjecture.  The
following corollary is perhaps not obvious, but follows in fact rather quickly
using an argument going back to Steenbrink and van Straten.

\begin{thm}[\protect{The Lipman-Zariski conjecture for klt spaces, \cite[Theorem~6.1]{GKKP11}}]\label{thm:LZ:1}
  Let $X$ be a normal, projective, klt variety.  If the tangent sheaf $\sT_X$ is
  locally free, then $X$ is smooth.  \qed
\end{thm}

We refer the reader to \cite[Section~6]{KP14}, which sketches a proof of
Theorem~\ref{thm:LZ:1} as a consequence of Theorem~\ref{thm:extpb}.  There are
generalisations as well as newer proofs that do not rely on the extension
theorem; cf.~\cite{Druel13a, MR3247804, MR3278896, MR3343876}.

\subsection{References}
\label{sssec:pb-ref}
\approvals{Behrouz & yes \\Daniel & yes \\Stefan & yes}

The universal properties of reflexive differentials on klt and log canonical
spaces were first established in the papers \cite{GKK08, GKKP11}.  The
formulation presented here comes from the subsequent paper \cite{MR3084424}.
The interested reader will definitively also want to look at \cite{MR3272910}
for a different take on the same circle of ideas.  The papers
\cite{ExtApplications, GKP11} as well as the surveys \cite{Keb13a, KP14} discuss
reflexive differentials and their applications in greater detail, see
\cite{Huber15} for a different perspective.  Kollár's book on the singularities
of the minimal model program also reviews the basic results,
\cite[Section~8.5]{MR3057950}.


%
%
\svnid{$Id: 04-gfa.tex 110 2016-08-18 11:27:16Z kebekus $}

\section{Existence of maximally quasi-étale covers}
\label{sec:qec}
\subversionInfo
\approvals{Behrouz & yes \\Daniel & yes \\Stefan & yes}

Representations of fundamental groups feature prominently in nonabelian Hodge
theory, and are one of the recurring themes in this survey,
cf.~Section~\ref{ssec:nahc} below.  If $X$ is smooth, projective and of
dimension $n := \dim X ≥ 3$, the classical Lefschetz hyperplane theorem allows
to reduce complexity by cutting down.  If $\sL ∈ \Pic(X)$ is very ample and
$H_1, …, H_{n-2} ∈ |\sL|$ are general hyperplanes with associated complete
intersection $S := H_1 ∩ \cdots ∩ H_{n-2}$, it asserts that the group morphism
induced by the inclusion, $π_1(S) → π_1(X)$, is isomorphic.  We refer to
\cite[Theorem~3.1.21]{Laz04-I} for a discussion.

The situation is substantially more involved when $X$ is singular, even in the
simple case where $X$ has isolated singularities, or somewhat more general,
where $X$ is smooth in codimension two ---this will be our most relevant
setting, since spaces with terminal singularities always have this property.
Under these assumptions, the general complete intersection surface $S$ is still
smooth and contained in the smooth locus $X_{\reg}$, but the appropriate
generalisation of the Lefschetz hyperplane theorem, \cite[Theorem in
Section~II.1.2]{GoreskyMacPherson}, only gives an isomorphism between $π_1(S)$
and $π_1(X_{\reg})$, rather than between $π_1(S)$ and $π_1(X)$.

In summary, we see that to use the cutting-down method successfully, we need to
compare $π_1(X_{\reg})$ and $π_1(X)$.  Since we are chiefly interested in
\emph{representations of fundamental groups} rather than fundamental groups
themselves, the following theorem of Grothendieck simplifies the problem
somewhat.

\begin{thm}[\protect{Profinite completions dictate representations, \cite[Theorem~1.2b]{MR0262386}}]\label{thm:gro1}
  Let $α : G → H$ be a morphism of between finitely generated groups, and let
  $α^* : \operatorname{Rep}_{ℂ}(H) → \operatorname{Rep}_{ℂ}(G)$ be the
  associated pull-back functor of finite-dimensional representation.  If the
  associated morphism $\what{α} : \what{G} → \what{H}$ between profinite
  completions is bijective, then $α^*$ induces an equivalence of categories.
  \qed
\end{thm}

For spaces with klt singularities, we have shown that the difference between
profinite completions $\what{π}_1(X)$ and $\what{π}_1(X_{\reg})$ can be made to
vanish.

\begin{thm}[\protect{Existence of maximally quasi-étale covers, \cite[Theorem~1.4]{GKP13}}]\label{thm:gfa:2}
  Let $X$ be a normal, complex, quasi-projective variety.  Assume that there
  exists a Weil $ℚ$-divisor $Δ$ such that $(X, Δ)$ is klt.  Then, there exists a
  normal variety $\wtilde X$ and a quasi-étale, Galois morphism
  $γ: \wtilde X → X$, such that the following, equivalent conditions hold.
  \begin{enumerate}
  \item\label{il:mc1} Any finite, étale cover of $\wtilde X_{\reg}$ extends to a
    finite, étale cover of $\wtilde X$.

  \item\label{il:mc2} The natural map
    $\what{ι}_* : \what{π}_1(\wtilde X_{\reg}) → \what{π}_1(\wtilde X)$ of étale
    fundamental groups induced by the inclusion of the smooth locus,
    $ι : \wtilde X_{\reg} → \wtilde X$, is an isomorphism.  \qed
  \end{enumerate}
\end{thm}

The proof of Theorem~\ref{thm:gfa:2} builds on work of Chenyang Xu who proved
that local étale fundamental groups vanish for spaces with isolated klt
singularities, \cite{Xu12}.  We want to emphasise that Xu's result is by no
means elementary, and uses many of the recent advances in higher-dimensional
birational geometry, such as boundedness results for log Fano manifolds.

\subsection{Application to flatness}
\approvals{Behrouz & yes \\Daniel & yes \\Stefan & yes}

We aim to apply Theorem~\ref{thm:gfa:2} to the study of flat sheaves on klt
spaces.  Since we are dealing with singular spaces, we do not attempt to define
flat sheaves via connections.  Instead, a flat sheaf $\sF$ will always be an
analytic, locally free sheaf, given by a representation of the fundamental
group.

\begin{defn}\label{defn:flat:1}
  If $Y$ is any complex space, and $\sG$ is any locally free sheaf on $Y$, we
  call $\sG$ \emph{flat} if it is defined by a representation of the fundamental
  group.  A locally free, algebraic sheaf on a complex algebraic variety $Y$ is
  called flat if and only if the associated analytic sheaf on the underlying
  complex space $Y^{an}$ is flat.
\end{defn}

We obtain the following consequences of Theorems~\ref{thm:gro1} and
\ref{thm:gfa:2}.

\begin{thm}[\protect{Étale fundamental groups dictate flatness, \cite[Section~11.1]{GKP13}}]\label{thm:flat:1}
  Let $X$ be a normal, complex, quasi-projective variety, and assume that the
  natural inclusion map between étale fundamental groups,
  $\what{ι}_* : \what{π}_1(X_{\reg}) → \what{π}_1(X)$, is isomorphic.  If $\sF°$
  is any flat, locally free, analytic sheaf defined on the complex manifold
  $X_{\reg}^{an}$, then there exists a flat, locally free, analytic sheaf $\sF$
  on $X^{an}$ such that $\sF° = \sF|_{X_{\reg}^{an}}$.
\end{thm}

\begin{thm}[\protect{Flat sheaves on maximally quasi-étale covers, \cite[Theorem~1.13]{GKP13}}]\label{thm:flat:2}
  Let $X$ be a normal, complex, quasi-projective variety.  Assume that there
  exists a Weil $ℚ$-divisor $Δ$ such that $(X, Δ)$ is klt.  Then, there exists a
  normal variety $\wtilde X$ and a quasi-étale, Galois morphism
  $γ: \wtilde X → X$, such that the following holds.  If $\sG°$ is any flat,
  locally free, analytic sheaf on the complex space $\widetilde X_{\reg}^{an}$,
  there exists a flat, locally free, algebraic sheaf $\sG$ on $\wtilde X$ such
  that $\sG°$ is isomorphic to the analytification of $\sG|_{\wtilde X_{\reg}}$.
  \qed
\end{thm}

Given a normal variety $X$ and a flat, locally free, analytic sheaf $\sF°$ on
$X_{\reg}^{an}$, Deligne has shown in \cite[II.5, Corollary~5.8 and
Theorem~5.9]{Deligne70} that $\sF°$ is algebraic, and thus extends to a
coherent, reflexive, algebraic sheaf $\sF$ on $X$.  The above theorems hence
provide criteria to guarantee that Deligne's extended sheaves are in fact
locally free.

\subsection{References}
\approvals{Behrouz & yes \\Daniel & yes \\Stefan & yes}

The existence of maximally quasi-étale covers has been shown in
\cite[Theorem~1.4]{GKP13}.  The paper contains more general results, discusses
the relation to flatness in details and gives applications.  The survey paper
\cite{KP14} covers these results in greater detail.


%
%
\svnid{$Id: 05-higgs-bdls.tex 124 2016-08-19 06:29:47Z kebekus $}

\section{Nonabelian Hodge theory}
\label{sec:naht}
\subversionInfo
\approvals{Behrouz & yes \\Daniel & yes \\Stefan & yes}

The proof of our main result makes heavy use of Simpson's nonabelian Hodge
correspondence, which relates representations of the fundamental group to Higgs
bundles.  We will also use Simpson's construction of variations of Hodge
structures from special Higgs bundles.  Before recalling these results in more
detail below, we begin with the definition of a Higgs bundle and present a few
examples.

\begin{defn}[Higgs bundle]\label{def:HiggsBdle}
  Let $X$ be a complex manifold.  A \emph{Higgs bundle} is a pair $(\sE,θ)$
  consisting of a holomorphic vector bundle $\sE$, together with an
  $\sO_X$-linear morphism $θ:\sE → \sE ⊗ Ω_X^1$, called \emph{Higgs field}, such
  that the composed morphism
  $$
  \xymatrix{ %
    \sE \ar[r]^(.35){θ} & \sE ⊗ Ω¹_X \ar[rr]^(.4){θ ⊗ \Id} && \sE ⊗ Ω¹_X ⊗ Ω¹_X
    \ar[rr]^(.55){\Id ⊗ Λ²} && \sE ⊗ Ω²_X }
  $$
  vanishes.  Following tradition, the composed morphism will be denoted by
  $θ Λ θ$.  A coherent subsheaf $\sF ⊆ \sE$ is said to be $θ$-invariant if
  $θ(\sF) ⊆ \sF⊗ Ω^1_X$.
\end{defn}

\begin{defn}[System of Hodge bundles]
  Let $X$ be a complex manifold.  A \emph{system of Hodge bundles} is a Higgs
  bundle $(\sE,θ)$ on $X$, together with a number $n ∈ ℕ$ and a direct sum
  decomposition
  $$
  \sE = \bigoplus_{p+q=n} \sE^{p,q}
  $$
  such that for all indices $(p,q)$, the restriction $θ|_{\sE^{p+q}}$ takes its
  image in $\sE^{p-1,q+1}⊗ Ω_X^1$.  The restricted maps are traditionally
  written as $θ^{p,q}:\sE^{p,q}→ \sE^{p-1,q+1}⊗ Ω_X^1$.
\end{defn}

\begin{example}[Higgs sheaves with trivial field]
  Let $X$ be a complex manifold.  Let $\sE$ be any holomorphic vector bundle.
  Then, together with and consider the zero morphism $θ : \sE → \sE ⊗ Ω¹_X$.  In
  this example, any subsheaf of $\sE$ is $θ$-invariant.
\end{example}

\begin{example}[A natural Higgs sheaf attached to a complex manifold]\label{ex:BQfield0}
  Let $X$ be a complex manifold.  Set $\sE := Ω¹_X ⊕ \sO_X$ and define an
  operator $θ$ as follows,
  $$
  \begin{matrix}
    θ : & \mathclap{Ω¹_X} & ⊕ & \mathclap{\sO_X} & \xrightarrow{\ \ \ \ \ } & \Bigl( Ω¹_X & ⊕ & \sO_X \Bigr) & ⊗ & Ω¹_X \\
    & a & + & b & \mapsto & ( 0 & + & 1 ) & ⊗ & a.
  \end{matrix}
  $$
  An elementary computation shows that $θ Λ θ = 0$, so that $(\sE, θ)$ forms a
  Higgs bundle.  Observe that the direct summand $\sO_X ⊆ \sE$ is $θ$-invariant.
  On the other hand, non-zero subsheaves of the direct summand $Ω¹_X$ are never
  invariant.  In fact, $(\sE, θ)$ a system of Hodge bundles.  Indeed, the
  corresponding direct sum decomposition is given by
  $$
  \sE = \sE^{1,0} ⊕ \sE^{0,1}, \quad\text{with}\quad \sE^{1,0} = Ω¹_X \text{ and
  } \sE^{0,1} = \sO_X.
  $$
\end{example}

\subsection{Elementary constructions}
\approvals{Behrouz & yes \\Daniel & yes \\Stefan & yes}

The Higgs bundles on a given complex manifold form a category, with the obvious
definition for a morphism.  The following additional constructions allow for
direct sums, tensor products, duals, and pulling-back.

\begin{construction}[Direct sum and tensor product]\label{cons:dstp}
  Let $X$ be a complex manifold, and let $(\sE_1, θ_1)$ and $(\sE_2, θ_2)$ be
  two Higgs bundles on $X$.  Then, there are natural Higgs fields on the direct
  sum and tensor product,
  $$
  (\sE_1 ⊕ \sE_2,\; θ_1 ⊕ θ_2) \quad\text{and}\quad (\sE_1 ⊗ \sE_2,\; θ_1 ⊗
  \Id_{\sE_2} + \Id_{\sE_1} ⊕ θ_2).
  $$
\end{construction}

\begin{construction}[Dual and endomorphisms]\label{cons:dl}
  Let $X$ be a complex manifold, and let $(\sE, θ)$ be a Higgs bundle on $X$.
  The Higgs field can be seen as a section of the sheaf $\sE^* ⊗ \sE ⊗ Ω¹_X$,
  which is naturally isomorphic to $\sE ⊗ \sE^* ⊗ Ω¹_X$.  This allows to equip
  $\sE^*$, and then also $\sEnd(\sE) = \sE^* ⊗ \sE$ with natural Higgs fields.
\end{construction}

\begin{construction}[Pull-back]\label{cons:pb1}
  Let $X$ be a complex manifold, and let $(\sE, θ)$ be a Higgs bundle on $X$.
  Given a morphism of manifolds, $f : Y → X$, consider the sheaf morphism $θ'$,
  defined as the composition of the following maps,
  $$
  f^* \sE \xrightarrow{f^* θ} f^* \bigl( \sE ⊗ Ω¹_X \bigr) = f^* \sE ⊗ f^* Ω¹_X
  \xrightarrow{\Id_{f^* \sE} ⊗ df} f^* \sE ⊗ Ω¹_Y.
  $$
  One verifies that $θ'Λ θ' = 0$, so that $θ'$ equips $f^* \sE$ with the
  structure of a Higgs bundle.
\end{construction}

\subsection{The nonabelian Hodge correspondence}
\label{ssec:nahc}
\approvals{Behrouz & yes \\Daniel & yes \\Stefan & yes}

The following major result of Simpson, known as the \emph{nonabelian Hodge
  correspondence} relates Higgs bundles to representations of the fundamental
group.  We define the appropriate notion of stability first.

\begin{defn}[Higgs stability]\label{def:hsta}
  Let $X$ be a projective manifold and $H ∈ \Div(X)$ be a nef divisor.  We say
  that $(\sE,θ)$ is \emph{semistable with respect to $H$} if the inequality
  $μ_H(\sF) ≤ μ_H(\sE)$ holds for all $θ$-invariant subsheaves with
  $0 < \rank \sF < \rank \sE$.  The Higgs sheaf is called \emph{stable} if the
  inequality is always strict.  Direct sums of stable Higgs bundles are called
  \emph{polystable}.
\end{defn}

\begin{thm}[\protect{Nonabelian Hodge correspondence, \cite[Theorem~1]{MR1159261}}]
  Let $X$ be a projective manifold and $H ∈ \Div(X)$ be an ample divisor.  Then,
  there exists an equivalence between the categories of all representations of
  $π_1(X)$, and of all $H$-semistable Higgs bundles with vanishing Chern
  classes.  \qed
\end{thm}

\subsection{Higgs bundles induced by variations of Hodge Structures}
\approvals{Behrouz & yes \\Daniel & yes \\Stefan & yes}

As mentioned in the introduction to this chapter, Simpson constructed variations of Hodge
structures via Higgs bundles.  We briefly recall the most relevant definitions
and results.

\begin{defn}[\protect{Polarised, complex variation of Hodge structures, \cite[Section~8]{MR944577}}]\label{def:pCVHS}
  Let $X$ be a complex manifold, and $w ∈ ℕ$ a natural number.  A
  \emph{polarised, complex variation of Hodge structures of weight $w$}, or
  \pCVHS\ in short, is a $\cC^\infty$-vector bundle $\cV$ with a direct sum
  decomposition $\cV = ⊕_{r+s=w} \cV^{r,s}$, a flat connection $D$ that
  decomposes as follows
  \begin{equation}\label{eq:Gtrans}
    D|_{\cV^{r,s}} : \cV^{r,s} → \cA^{0,1}(\cV^{r+1,s-1}) ⊕ \cA^{1,0}(\cV^{r,s}) ⊕
    \cA^{0,1}(\cV^{r,s}) ⊕ \cA^{1,0}(\cV^{r-1,s+1}),
  \end{equation}
  and a $D$-parallel Hermitian metric on $\cV$ that makes the direct sum
  decomposition orthogonal and that on $\cV^{r,s}$ is positive definite if $r$
  is even and negative definite if $r$ is odd.
\end{defn}

Given a \pCVHS, one constructs an associated Higgs bundle as follows.

\begin{construction}[Higgs sheaves induced by a \pCVHS]\label{cons:hs}
  Given a \pCVHS\ as in Definition~\ref{def:pCVHS}, use \eqref{eq:Gtrans} to
  decompose $D$ as $D = \overline{θ} ⊕ \partial ⊕ \overline{\partial} ⊕ θ$.  The
  operators $\overline{\partial}$ equip the $\cC^\infty$-bundles $\cV^{r,s}$
  with complex structures.  We write $\sE^{r,s}$ for the associated locally free
  sheaves of $\sO_X$-modules, and set $\sE := ⊕ \sE^{r,s}$.  The operators $θ$
  then define an $\sO_X$-linear morphism $\sE → \sE ⊗ Ω^1_X$.  As $D$ is flat,
  this is a Higgs field.
\end{construction}

\begin{defn}[Higgs bundles induced by a \pCVHS]\label{def:hscvhs}
  Let $X$ be a complex manifold and $(\sE, θ)$ a Higgs bundle on $X$.  We say
  that \emph{$(\sE, θ)$ is induced by a \pCVHS} if there exists a \pCVHS\ on $X$
  such that $(\sE, θ)$ is isomorphic to the Higgs bundle obtained from it via
  Construction~\ref{cons:hs}.
\end{defn}

Scaling the Higgs field induces an action of $ℂ^*$ on the set of isomorphism
classes of Higgs bundles.  Under suitable assumptions, Simpson shows that Higgs
bundles induced by a \pCVHS\ correspond exactly to $ℂ^*$-fixed points.  The
following theorem summarises his results.

\begin{thm}[{Higgs bundles induced by a \pCVHS, I, \cite[Corollary~4.2]{MR1179076}}]\label{thm:charPCVHS}
  Let $X$ be a complex, projective manifold of dimension $n$ and $H ∈ \Div(X)$
  be an ample divisor.  Let $(\sE, θ)$ be a Higgs bundle on $X$.  Then,
  $(\sE, θ)$ comes from a variation of Hodge structures in the sense of
  Definition~\ref{def:hscvhs} if and only if the following three conditions
  hold.
  \begin{enumerate}
  \item\label{il:emerson} The Higgs bundle $(\sE, θ)$ is $H$-polystable.
  \item\label{il:lake} The intersection numbers $ch_1(\sE)·[H]^{n-1}$ and
    $ch_2(\sE)·[H]^{n-2}$ both vanish.
  \item\label{il:palmer} For any $t ∈ ℂ^*$, the Higgs bundles $(\sE, θ)$ and
    $(\sE, t·θ)$ are isomorphic.  \qed
  \end{enumerate}
\end{thm}

\begin{rem}\label{rem:charPCVHS}
  With $X$ and $H$ as in Theorem~\ref{thm:charPCVHS}, any Higgs bundle
  $(\sE, θ)$ that satisfies \ref{il:emerson} and \ref{il:lake} carries a flat
  $\cC^\infty$-connection, \cite[Theorem~1(2) and Corollary~1.3]{MR1179076}.  In
  particular, all its Chern classes vanish.
\end{rem}

As one immediate consequence of Theorem~\ref{thm:charPCVHS}, we obtain the
following minor strengthening of \cite[Corollary~4.3]{MR1179076}.

\begin{cor}[\protect{Higgs bundles induced by a \pCVHS, II, \cite[Corollary~6.36]{GKPT15}}]\label{cor:pfccs}
  Let $X$ be a projective manifold, and $H ∈ \Div(X)$ be an ample divisor.  Let
  $\imath: S \hookrightarrow X$ be a submanifold.  The push-forward map
  $\imath_*: π_1(S) → π_1(X)$ induces a restriction map
  $$
  \begin{array}{rccc}
    r : & \begin{Bmatrix}
      \text{Isomorphism classes of $H$-semi-}\\
      \text{stable Higgs bundles $(\sE, θ)$ on $X$} \\
      \text{with vanishing Chern classes.}
    \end{Bmatrix}
    & → &
    \begin{Bmatrix}
      \text{Isomorphism classes of $H$-semi-}\\
      \text{stable Higgs bundles $(\sE, θ)$ on $S$} \\
      \text{with vanishing Chern classes.}
    \end{Bmatrix} \\[0.6cm]
    & (\sE, θ) & \mapsto & (\sE, θ)|_S.
  \end{array}
  $$
  In particular, if $(\sE, θ)$ is any $H$-semistable Higgs bundle $(\sE, θ)$ on
  $X$ with vanishing Chern classes, then $(\sE, θ)|_S$ is again $H$-semistable.
  The map $r$ has the following properties.
  \begin{enumerate}
  \item\label{il:simon} If $\imath_*$ is surjective, then $r$ is injective.  In
    particular, if $(\sE, θ)$ is a Higgs bundle on $X$ such that $(\sE, θ)|_S$
    comes from a \pCVHS, then $(\sE, θ)$ comes from a \pCVHS.
  \item\label{il:garfunkel} If in addition the induced push-forward map
    $\what{\imath}_*:\what{π}_1(S) → \what{π}_1(X)$ of algebraic fundamental
    groups is isomorphic, then $r$ is surjective.
  \end{enumerate}
\end{cor}
\begin{proof}
  Simpson's nonabelian Hodge correspondence, Theorem~\ref{cons:hs}, gives an
  equivalence between the categories of representations of the fundamental group
  $π_1(X)$ (resp.\ $π_1(S)$) and $H$-semistable Higgs bundles on $X$ (resp.\
  $S$) with vanishing Chern classes.  The correspondence is functorial in
  morphisms between manifolds, and pull-back of Higgs bundles corresponds to the
  push-forward of fundamental groups, \cite[Remark~1 on Page~36]{MR1179076}.  In
  particular, we see that the restriction of an $H$-semistable Higgs bundle with
  vanishing Chern classes is again $H$-semistable.

  In the setting of \ref{il:simon} where the push-forward map $π_1(S) → π_1(X)$
  is surjective, this immediately implies that the restriction $r$ is injective.
  The restriction map $r$ is clearly equivariant with respect to the actions of
  $ℂ^*$ obtained by scaling the Higgs fields.  Injectivity therefore implies
  that the isomorphism class of a Higgs bundle $(\sE, θ)$ is $ℂ^*$-fixed if and
  only if the same is true for $(\sE, θ)|_S$.  Theorem~\ref{thm:charPCVHS} thus
  proves the second clause of \ref{il:simon}.

  Now assume that we are in the setting of \ref{il:garfunkel}, where in addition
  the push-forward map $\what{π}_1(S) → \what{π}_1(X)$ is assumed to be
  isomorphic.  Since fundamental groups of algebraic varieties are finitely
  generated, Theorem~\ref{thm:gro1} implies that every representation of
  $π_1(S)$ comes from a representation of $π_1(X)$.  The claim thus again
  follows from Simpson's nonabelian Hodge correspondence.
\end{proof}

The following proposition links Higgs bundles coming from variations of Hodge
structures to minimal model theory.  It is crucial for all that follows.

\begin{prop}[\protect{Higgs bundles and minimal model theory, \cite[Corollary~6.39]{GKPT15}}]\label{prop:higgsfromdownst}
  Let $Y$ be a normal, projective variety with at worst canonical singularities
  and let $π : \wtilde Y → Y$ be a resolution of singularities.  Let
  $(\sF_{\wtilde Y}, θ_{\wtilde Y})$ be a Higgs bundle on $\wtilde Y$ that is
  induced by a \pCVHS.  Then, $\sF_{\wtilde Y}$ comes from $Y$.  More precisely,
  there exists a locally free sheaf $\sF_Y$ on $Y$ such that
  $\sF_{\wtilde Y} = π^* \sF_Y$.  Necessarily, we then have
  $\sF_{\wtilde Y} \cong π_* (\sF_Y)^{**}$.
\end{prop}
\begin{proof}
  It suffices to construct $\sF_Y$ locally in the analytic topology, near any
  given point $y ∈ Y$.  Recall from \cite[Page~827]{Takayama2003} that there exists a contractible, open
  neighbourhood $U = U(y) ⊆ Y^{an}$ whose preimage $\wtilde U := π^{-1}(U)$ is
  simply connected.  By assumption, $(\sF_{\wtilde Y}, θ_{\wtilde Y})$ is
  induced from a \pCVHS{}, say $\cV$.  Let $ρ : \wtilde U → \cD$ be the
  corresponding period map.
  
  We claim that $ρ$ factors through the resolution,
  $$
  \xymatrix{ %
    \wtilde U \ar[r]_{π} \ar@/^3mm/[rr]^{ρ} & U \ar[r]_{\exists\; ρ_U} & \cD.
  }
  $$
  Since $U$ is normal, this will follows once we show that $ρ$ is constant on
  fibres of $π$.  The fibres of $π$, however, are known to be rationally
  chain-connected, \cite[Corollary~1.5]{HMcK07}.  In summary, $ρ$ will factor as
  soon as we show that for any morphism $η: ℙ¹ → \wtilde U$, the composed map
  $ρ ◦ η : ℙ¹ → \cD$ is constant.  Given one such $η$, we obtain a \pCVHS{} on
  $ℙ¹$ whose associated period map equals $ρ◦η$, simply by pulling back $\cV$
  via $η$.  However, due to hyperbolicity properties of the period domain $\cD$,
  this map has to be constant, \cite[Application 13.4.3]{CMSP}.
    
  It is known that $\sF_{\wtilde Y}|_{\wtilde U} \cong ρ^*(\sF_{\cD})$ for some
  vector bundle $\sF_{\cD}$ on the period domain $\cD$,
  cf.~\cite[Proposition~6.38]{GKPT15}.  If $ρ_U: U → \cD$ is the holomorphic map
  whose existence was shown in the previous paragraph, the vector bundle
  $\sF_U := ρ_U^*(\sF_{\cD})$ hence fulfils
  $π^*(\sF_U) \cong \sF_{\wtilde Y}|_{\wtilde U}$, as desired.
\end{proof}

\subsection{References}
\approvals{Behrouz & yes \\Daniel & yes \\Stefan & yes}

Higgs operators appeared in~\cite{Hit87} where Hitchin studied Yang-Mills
equations with the aim of finding conditions for existence of flat connections
on a compact Riemann surface.  In analogy to results of Narasimhan-Seshadri,
Hitchin observed that solutions to Yang-Mills equations impose additional
holomorphic data on the given holomorphic bundle, a condition that is nowadays
refereed to as Higgs stability.  Higgs fields were also introduced in the theory
of variation of Hodge structures in smooth families of projective varieties,
where they are encoded in the Griffiths transversality and holomorphicity
properties of the Gauss-Manin connection.  A fundamental result of Griffiths,
cf.~\cite{Gri68}, then showed that the existence of variation of Hodge
structures gives rise to a holomorphic map, the \emph{period map}, from the
universal cover to the classifying space of Hodge structures.

This result of Griffiths led Simpson to study uniformisation problems via
variations of Hodge structures.  He aimed to find holomorphic and numerical
conditions on a suitable Higgs bundle for it to define a complex variation of
Hodge structures whose associated period map would then provide an
identification of the universal cover.  This was famously achieved
in~\cite{MR944577}.  The arguments are parallel to earlier work of Hitchin,
Donaldson, and Uhlenbeck-Yau, \cite{Hit87, Donaldson85, UhlenbeckYau86}, in
tracing a correspondence between stable Higgs bundles with vanishing Chern
classes and flat connections.

There are many overview papers on the subject, including Simpson's ICM talk
\cite{MR1159261}.  The reader might also want to look at the excellent survey
\cite{MR3088903}, or at the short note \cite{MR2343296}.


%
%
\svnid{$Id: 06-higgs-sheaves.tex 116 2016-08-18 12:12:37Z kebekus $}

\section{Higgs sheaves on singular spaces}
\label{sec:Higgs}
\subversionInfo

\subsection{Fundamentals}
\approvals{Behrouz & yes \\Daniel & yes \\Stefan & yes}

On a singular variety, the correct definition of a ``Higgs sheaf'' is perhaps
not obvious.  As we will note below, the following generalisation of
Definition~\ref{def:HiggsBdle} turns out to have just enough universal
properties to make the strategy of our proof work.  In the converse direction,
it seems that Definition~\ref{def:Higgs} and our notion of stability are in
essence uniquely dictated if we ask all these universal properties to hold.

\begin{defn}[Higgs sheaf, generalisation of Definition~\ref{def:HiggsBdle}]\label{def:Higgs}
  Let $X$ be a normal variety.  A \emph{Higgs sheaf} is a pair $(\sE, θ)$
  consisting of a coherent sheaf $\sE$ of $\sO_X$-modules, together with an
  $\sO_X$-linear morphism $θ : \sE → \sE ⊗ Ω^{[1]}_X$, called \emph{Higgs
    field}, such that the composed morphism
  $$
  \xymatrix{ %
    \sE \ar[r]^(.35){θ} & \sE ⊗ Ω^{[1]}_X \ar[rr]^(.45){θ ⊗ \Id} && \sE ⊗
    Ω^{[1]}_X ⊗ Ω^{[1]}_X \ar[rr]^(.55){\Id ⊗ [Λ]} && \sE ⊗ Ω^{[2]}_X }
  $$
  vanishes.  Following tradition, the composed morphism will be denoted by
  $θ Λ θ$.  The definition of \emph{systems of Hodge sheaves} carries over
  verbatim.
\end{defn}

\begin{warning}
  There exists an obvious notion of \emph{morphism of Higgs sheaves}, but there
  is generally no way to equip kernels or cokernels with Higgs fields.  Higgs
  sheaves hence do not form an Abelian category.
\end{warning}

\begin{defn}[Invariant and generically invariant subsheaves]\label{def:Higgsinvar}
  Setting as in Definition~\ref{def:Higgs}.  A coherent subsheaf $\sF ⊆ \sE$ is
  called \emph{$θ$-invariant} if $θ(\sF)$ is contained in the image of the
  natural map
  $$
  \sF ⊗ Ω^{[1]}_X → \sE ⊗ Ω^{[1]}_X.
  $$
  Call $\sF$ \emph{generically invariant} if the restriction $\sF|_{X_{\reg}}$
  is invariant with respect to $θ|_{X_{\reg}}$.
\end{defn}

\begin{warning}
  As $Ω_X^{[1]}$ is not locally free, in Definition~\ref{def:Higgsinvar} the
  sheaf $\sF ⊗ Ω^{[1]}_X$ is generally not a subsheaf of $\sE ⊗ Ω^{[1]}_X$.  As
  a consequence, there is generally no induced Higgs field on invariant or
  generically invariant subsheaves.  At this point, our setting differs
  substantially from the smooth case.  Even though we will later define
  (semi-)stability for Higgs sheaves, this will make it impossible to easily
  construct an analogue of the Harder-Narasimhan filtration.
\end{warning}

\subsection{Explanation and examples}
\approvals{Behrouz & yes \\Daniel & yes \\Stefan & yes}

At first sight, it might seem most natural and functorial to define Higgs fields
as morphisms to $\sE ⊗ Ω_X^1$.  However, in our application to uniformisation
for varieties of general type, the naturally induced sheaf of geometric origin
is $\sE := Ω^{[1]}_X ⊕ \sO_X$, as discussed in Example~\ref{ex:BQfield1} below.
Looking at $Ω_X^1 ⊕ \sO_X$ instead would render any discussion of semistability
moot, as semistability requires torsion freeness and even the most simple klt
singularities lead to torsion in $Ω_X^1$, see \cite{MR2915479} for examples.
  
On the other hand, the reader might wonder why $θ$ takes its values in
$\sE ⊗ Ω^{[1]}_X$ and not in its reflexive hull.  The advantages of our choice
will become apparent when pull-back functors are defined.  None of the
constructions there will work for reflexive hulls.

\begin{example}[A natural Higgs sheaf attached to a normal variety, generalising Example~\ref{ex:BQfield0}]\label{ex:BQfield1}
  Let $X$ be a normal variety.  Set $\sE := Ω^{[1]}_X ⊕ \sO_X$ and define a
  Higgs field
  $$
  θ : \sE → \sE⊗Ω^{[1]}_X, \quad a + b \mapsto (0+1)⊗a.
  $$
  As before, the direct summand $\sO_X ⊆ \sE$ is generically $θ$-invariant, and
  subsheaves of the direct summand $Ω^{[1]}_X$ are never generically
  $θ$-invariant, unless they are zero.
\end{example}

\begin{construction}[Direct sum, tensor product, dual and endomorphisms]
  Construction~\ref{cons:dstp} of Higgs fields on the direct sum and tensor
  product of two Higgs bundles carries over to Higgs sheaves.  If the Higgs
  sheaf is locally free, an immediate analogue of Construction~\ref{cons:dl}
  defines natural Higgs fields on the dual sheaf and on the endomorphism sheaf.
  We refer to \cite[Sections~5.1 and 6.1]{GKPT15} for details, and for further
  constructions.
\end{construction}

\subsection{Pull-back}
\label{sec:pull-back}
\approvals{Behrouz & yes \\Daniel & yes \\Stefan & yes}

One of the most basic properties of Higgs bundles is the existence of a
pull-back functor.  For Higgs sheaves on singular spaces, we do not believe that
a reasonable notion of pull-back exists in general.  In fact, to pull back Higgs
sheaves is at least as difficult as to pull-back reflexive differentials, and
examples abound which show that there is generally no notion of pulling-back for
reflexive differentials.  Worse still, even in settings where pull-back
morphisms happen to exist, the pull-back may fail to be functorial.  For spaces
with klt singularities, however, we have seen in Section~\ref{sec:reflDiff} that
functorial pull-back functor does exist.  For these spaces, the following
construction will then give a functorial pull-back of Higgs sheaves.

\begin{construction}[Pull-back of Higgs sheaves, generalisation of Construction~\ref{cons:pb1}]\label{cons:pb2}
  Let $(X,D)$ be a klt pair and let $(\sE, θ)$ be a Higgs sheaf on $X$.  Given a
  normal variety $Y$ and a morphism $f : Y → X$, consider the sheaf morphism
  $θ'$, defined as the composition of the following maps,
  $$
  f^* \sE \xrightarrow{f^* θ} f^* \Bigl( \sE ⊗ Ω^{[1]}_X \Bigr) = f^*
  \sE ⊗ f^* Ω^{[1]}_X \xrightarrow{\Id_{f^* \sE} ⊗
    \drefl f} f^* \sE ⊗ Ω^{[1]}_Y.
  $$
  One verifies that $θ'Λ θ' = 0$, so that $θ'$ equips $f^* \sE$ with the
  structure of a Higgs sheaf.  By minor abuse of notation, this Higgs sheaf will
  be denoted as $f^* (\sE, θ)$ or $(f^* \sE, f^* θ)$.  If $f$ is a closed or
  open immersion, we will also write $(\sE, θ)|_Y$ or $(\sE|_Y, θ|_Y)$.
\end{construction}

If the space $Y$ of Construction~\ref{cons:pb2} is smooth, the construction can
be generalised further, to define a Higgs field on the reflexive pull-back
$f^{[*]}\sE := \bigl( f^* \sE \bigr)^{**}$.  The resulting notion of ``reflexive
pull-back'' is important, but fails to have any form of functoriality,
cf.~\cite[Sect.~6.4]{GKPT15}.

\subsection{Stability}
\label{ssect:restrict}
\approvals{Behrouz & yes \\Daniel & yes \\Stefan & yes}

We close this section generalising the notion of stability from Higgs bundles to
Higgs sheaves.  Again, it might not be obvious at first sight that the following
definition, which considers slopes of subsheaves that are only generically
injective, is the ``right'' one.  It has the advantage that it behaves well with
respect to the reflexive pull-back discussed above.  The paper \cite{GKPT15}
uses this to compare stability of the Higgs sheaf $(\sE, θ)$ with that of its
reflexive pull-back.

\begin{defn}[Stability of Higgs sheaves]\label{defn:swostab1}
  Let $X$ be a normal, projective variety and $H$ be any nef, $ℚ$-Cartier
  $ℚ$-divisor on $X$.  Let $(\sE, θ)$ be a Higgs sheaf on $X$, were $\sE$ is
  torsion free.  We say that $(\sE, θ)$ is \emph{semistable with respect to $H$}
  if the inequality $μ_H(\sF) ≤ μ_H(\sE)$ holds for all generically
  $θ$-invariant subsheaves $\sF ⊆ \sE$ with $0 < \rank \sF < \rank \sE$.  The
  Higgs sheaf is called \emph{stable with respect to $H$} if the inequality is
  always strict.  Direct sums of stable Higgs sheaves are called
  \emph{polystable}.
\end{defn}

\begin{rem}
  For Higgs bundles, Definition~\ref{defn:swostab1} reproduces the earlier
  notion of stability, as introduced in Definition~\ref{def:hsta} above.  We
  refer to \cite[Sect.~6.6]{GKPT15} for details.
\end{rem}

\subsubsection{The restriction theorem}
\approvals{Behrouz & yes \\Daniel & yes \\Stefan & yes}

We conclude with a restriction theorem of Mehta-Ramanathan type, which will be
crucial for the proof of our main results.  Its (rather long and protracted)
proof relies on Langer's generalised Bogomolov-Gieseker inequalities for sheaves
with operators, resolving singularities and cutting down in order to reduce to a
setting where Langer's results apply.  The functorial properties of Higgs
sheaves play a pivotal role in this.

\begin{thm}[\protect{Restriction theorem for stable Higgs sheaves, \cite[Theorem~6.22]{GKPT15}}]\label{thm:restriction}
  Let $(X,Δ)$ be a projective klt pair of dimension $n ≥ 2$, let $H ∈ \Div(X)$
  be an ample, $ℚ$-Cartier $ℚ$-divisor and let $(\sE, θ)$ be a torsion free
  Higgs sheaf on $X$ of positive rank.  Assume that $(\sE, θ)$ is stable with
  respect to $H$.  If $m \gg 0$ is sufficiently large and divisible, then there
  exists a dense open set $U ⊆ |m·H|$ such that the following holds for any
  hyperplane $D ∈ U$ with associated inclusion map $ι : D → X$.
  \begin{enumerate}
  \item The hyperplane $D$ is normal, connected and not contained in $\supp Δ$.
    The pair $(D,Δ|_D)$ is klt.
  \item The sheaf $\sE|_D$ is torsion free.  The Higgs sheaf $ι^{*}(\sE, θ)$ is
    stable with respect to $H|_D$.  \qed
  \end{enumerate}
\end{thm}

For Higgs bundles on manifolds with ample polarisation, the theorem appears in
Simpson's work, \cite[Lemma~3.7]{MR1179076}.  Langer proves a similar theorem
for sheaves on projective manifolds, polarised by tuples of divisors that need
not be ample, \cite[Theorem~10]{MR3314517}.  He works in positive characteristic
but says that \emph{mutatis mutandis}, his arguments will also work in
characteristic zero, cf.\ \cite[Page~906]{MR3314517}.


\part{Proof of the main results}
%
%
\svnid{$Id: 07-torus.tex 116 2016-08-18 12:12:37Z kebekus $}

\section{Characterisation of torus quotients}
\approvals{Behrouz & yes \\Daniel & yes \\Stefan & yes}
\label{ssec:toriq}

In this section we will very briefly sketch the proof of Theorem~\ref{thm:TQ} on
the uniformisation of singular varieties with vanishing Chern classes by the
Euclidean space.  There are various similarities and some crucial differences
between the methods required for the proof of the two uniformisation results,
Theorems~\ref{thm:TQ} and \ref{thm:BQ}.  Our hope is that a comparison between
the two proofs would prove useful in clarifying the main ideas and techniques
behind both results.  We have therefore chosen to present an outline of the
proof following the strategy of~\cite{GKP13}, even though this is covered in at
least one other survey, \cite[Section~9]{KP14}.  We remark that the case of
canonical threefolds with vanishing Chern classes was achieved by
Shepherd-Barron and Wilson in~\cite{SBW94}.  Theorem~\ref{thm:TQ} has been
generalised to klt spaces in \cite{LT14}, providing a complete numerical
characterisation of quotients of Abelian varieties by finite groups acting
freely in codimension one.  Both of these latter results require working with
orbifold Chern classes\footnote{or ``$ℚ$-Chern classes''} which would require a
rather lengthy preparation and technical details that, for the sake of
simplicity, we have decided to avoid in the current article.

\subsection{Outline of the proof of Theorem~\ref*{thm:TQ}}\label{subsect:torusproof}
\approvals{Behrouz & yes \\Daniel & yes \\Stefan & yes}

The proof consists of two main steps.  Our aim in the first step, which is
modelled on the strategy of \cite{MR84}, is to construct a reflexive sheaf $\sF$
on $X$, formed as the coherent extension of a flat, locally-free, analytic sheaf
on $X_{\reg}$, that verifies the isomorphism $\sF|_S\cong \sT_X|_S$, for a
complete intersection surface $S$ cut out by general members of linear systems
of sufficiently large multiples of $H$.  In the second step we use the
aforementioned sheaf isomorphism on $S$ to find a global isomorphism
$\sF \cong \sT_X$.  Of course when $X$ is smooth, this already implies that
$\sT_X$ is flat.  When $X$ is singular, one then needs a method to extend the
flatness of $\sT_{X_{\reg}}$ across the singular locus.  According to
Theorem~\ref{thm:flat:2} this can be achieved when the singularities are mild,
at least up to a suitable cover.  This is the main ingredient of the second
step.

\subsubsection*{Step~1: Construction of a flat sheaf on $X_{\reg}$}
\approvals{Behrouz & yes \\Daniel & yes \\Stefan & yes}

We first notice that owing to the celebrated generic semipositivity result of
Miyaoka~\cite{Miyaoka87} we know that $\sT_X$ is slope-semistable with respect
to $H$.  Next, choose a sufficiently large and divisible integer $m \gg 0$, and
choose a general tuple of hyperplanes $D_1, …, D_{n-2} ∈ |m·H|$, with general
complete intersection surface $S := D_1 ∩ \cdots ∩ D_{n-2}$.  The following
items will then hold.

\CounterStep
\begin{enumerate}
\item\label{il:whitehorse0} The intersection $S$ is a smooth surface, and
  entirely contained in $X_{\reg}$.  This is because $X$ is smooth in
  codimension two by assumption.

\item\label{il:redhorse0} The restriction $\sT_X|_S$ is semistable with respect
  to $H|_S$.  This follows from Flenner's Mehta-Ramanathan theorem for normal
  varieties, \cite[Theorem~1.2]{Flenner84}.

\item\label{il:blackhorse0} The natural morphism $ι_*: π_1(S) → π_1(X_{\reg})$,
  induced by the inclusion $ι: S \into X_{\reg}$, is isomorphic.  This is the
  content of Goresky-MacPherson's Lefschetz hyperplane sheorem for homotopy
  groups, \cite[Theorem in Section~II.1.2]{GoreskyMacPherson}.

\item\label{il:dappledhorse0} Let $\sF°$ be any locally free, flat, analytic
  sheaf on $X_{\reg}$ with $\rank \sF = n$.  Then, $\sF°$ is isomorphic to
  $\sT_{X_{\reg}}$ if and only if the restrictions $\sF°|_S$ and
  $\sT_{X_{\reg}}|_S$ are isomorphic.  This follows because flat sheaves of
  fixed rank form a bounded family, \cite[Proposition~9.1]{GKP13}, and because
  of the Bertini-type theorem for isomorphism classes in bounded families,
  \cite[Corollary~5.3]{GKP13}.
\end{enumerate}
Now according to~\cite[Corollary~3.10]{MR1179076} the semistability of
$\sT_X|_S$ together with the vanishing condition on its Chern classes imply that
$\sT_X|_S$ comes from a representation of $π_1(S)$.  Item~\ref{il:blackhorse0}
now allows to extend this to a representation of $π_1(X_{\reg})$.  In other
words, we find a locally-free, flat bundle $\sF°$ on $X_{\reg}$ such that
$\sF°|_S \cong \sT_X|_S$.  Define $\sF := ι_* \sF°$, where $ι : X_{\reg} → X$ is
the inclusion map.  The sheaf $\sF$ is then coherent, and in fact reflexive on
$X$.

\subsubsection*{Step~2: Reduction to the smooth case}
\approvals{Behrouz & yes \\Daniel & yes \\Stefan & yes}

As $\sF°$ is flat, Item~\ref{il:dappledhorse0} applies and we find that
$\sF° \cong \sT_{X_{\reg}}$.  Now, let $γ: \wtilde X → X$ be a maximal
quasi-étale cover, as given by Theorem~\ref{thm:flat:2}.  Since $γ$ is
unramified in codimension one, $\sT_{\wtilde X} \cong γ^{[*]}(\sT_X)$.  As one
consequence, we see that $\sT_{\wtilde X}$ is flat over
$\wtilde X° := γ^{-1}(X_{\reg})$, which is a big subset of $\wtilde X$.  We also
see that
$$
K_{\wtilde X}\equiv 0 \quad\text{and}\quad c_2 \bigl(\sT_{\wtilde X})·
[γ^*H]^{n-2}=0,
$$
where the last equality is a consequence of the projection formula.  According
to Theorem~\ref{thm:flat:2}, the sheaf $\sT_{\wtilde X}$ comes from a
representation of $π_1(\wtilde X)$.  In particular, it is locally-free.  Now
thanks to the solution to the Lipman-Zariski conjecture, Theorem~\ref{thm:LZ:1},
we find that $\wtilde X$ is smooth.  Theorem~\ref{thm:TQ} now follows from the
original result of Yau, Theorem~\ref{thm:YausTheorem}.  \qed


%
%
\svnid{$Id: 08-MY.tex 116 2016-08-18 12:12:37Z kebekus $}

\section{Proof of the Miyaoka-Yau inequality}
\label{sec:MY}

\subsection{Proof of Theorem~\ref*{thm:MYinequality} in a simplified setting}
\approvals{Behrouz & yes \\Daniel & yes \\Stefan & yes}

For the purposes of this survey, we prove Theorem~\ref{thm:MYinequality} only
under the following simplifying assumptions.  Section~\ref{ssec:pigc1} briefly
discusses the missing pieces for a proof in the general case.

\begin{assumption}\label{ass:simpl1}
  The canonical bundle $K_X$ is ample, and $X$ is smooth in codimension two.  In
  particular, Chern classes $c_1$ and $c_2$ exist.
\end{assumption}

\subsubsection*{Step 1: Setup.  The natural Higgs sheaf on $X$}
\approvals{Behrouz & yes \\Daniel & yes \\Stefan & yes}

We begin by considering the natural Higgs sheaf $(\sE, θ)$, as given in
Example~\ref{ex:BQfield1}, where $\sE = Ω^{[1]}_X ⊕ \sO_X$ and
$θ(a + b) = (0+1)⊗a$.  The main reason for our interest in $(\sE, θ)$ is the
observation that the Bogomolov-Gieseker discriminant $Δ(\sE)$ computes the
Miyaoka-Yau discriminant of $\sT_X$.  Indeed, we have
\begin{align*}
  Δ(\sE) · [K_X]^{n-2} & := \bigl(2(\rank \sE) · c_2(\sE_X) - ((\rank \sE)-1) · c_1^2(\sE_X)\bigr) · [K_X]^{n-2} \\
  & = \bigl(2(n+1) · c_2(\sT_X) - n · c_1^2(\sT_X) \bigr) · [K_X]^{n-2}.
\end{align*}
To establish the Miyaoka-Yau inequality for $\sT_X$, it will therefore suffice
to show that $\sE$ verifies the Bogomolov-Gieseker inequality, $Δ(\sE) ≥ 0$.
This will follow from a major result of Simpson, who verified the
Bogomolov-Gieseker inequality for Higgs bundles that are stable with respect to
an ample polarisation, \cite[Theorem~1 and Proposition~3.4]{MR944577}.  To apply
Simpson's result, we need to show that $(\sE,θ)$ is stable with respect to
$K_X$, and then cut down to reduce to the case of a Higgs bundle (rather than a
mere sheaf) on a smooth surface.

\subsubsection*{Step 2: Stability}
\approvals{Behrouz & yes \\Daniel & yes \\Stefan & yes}

Generalising a classical result of Enoki, \cite[Corollary~1.2]{Eno87}, Guenancia
\cite[Theorem~A]{Guenancia} has shown that the tangent sheaf of a klt projective
variety with ample canonical sheaf is necessarily polystable.  Projecting a
potentially destabilising, generically $θ$-invariant subsheaf of $(\sE, θ)$ to
the $\sO_X$-summand of $\sE$ and recalling from Example~\ref{ex:BQfield1} that
no subsheaf of the direct summand $Ω^{[1]}_X$ is ever generically $θ$-invariant,
we deduce the following result, see \cite[Corollary 8.2]{GKPT15}.

\begin{lem}\label{lem:1}
  The Higgs sheaf $(\sE,θ)$ is stable with respect to $K_X$.  \qed
\end{lem}

\subsubsection*{Step 3: End of proof}
\approvals{Behrouz & yes \\Daniel & yes \\Stefan & yes}

Choose a sufficiently large and divisible integer $m \gg 0$, and choose a
general tuple of hyperplanes $H_1, …, H_{n-2} ∈ |m·K_X|$, with general complete
intersection surface $S := H_1 ∩ \cdots ∩ H_{n-2}$.  Using the assumption that
$X$ is smooth in codimension two, the surface $S$ is smooth, and entirely
contained in the smooth locus of $X$.  In particular, $(\sE,θ)|_S$ is a Higgs
bundle.  A repeated application of the restriction theorem for stable Higgs
sheaves, Theorem~\ref{thm:restriction}, shows that $(\sE,θ)|_S$ is stable with
respect to $K_X|_S$, and Simpson's result \cite[Theorem~1 and
Proposition~3.4]{MR944577} applies to give that
$$
0 ≤ Δ(\sE|_S) = \frac{Δ(\sE|_S)·[K_X]^{n-2}}{m^{n-2}}.
$$
As we have seen in Step~1, this finishes the proof of
Theorem~\ref{thm:MYinequality} in the simplified setting of
Assumption~\ref{ass:simpl1}.  \qed

\subsection{Proof of Theorem~\ref*{thm:MYinequality} in the general case}
\approvals{Behrouz & yes \\Daniel & yes \\Stefan & yes}
\label{ssec:pigc1}

The proof in the general case works along the same lines as the proof presented
above.  However, there are two problems that need to be overcome.

\subsubsection{The canonical sheaf might not be ample}
\approvals{Behrouz & yes \\Daniel & yes \\Stefan & yes}

By assumption, the canonical divisor $K_X$ is nef and not necessarily ample.  It
can, however, be approximated by ample divisors.  This adds an additional layer
of complexity but causes no fundamental problems, because Simpson's theory works
with arbitrary ample divisors, which may or may not equal $K_X$.

\subsubsection{The variety is not necessarily smooth in codimension two}
\approvals{Behrouz & yes \\Daniel & yes \\Stefan & yes}

The proof presented above used that assumption that $X_{\reg}$ is a big set.
That is not necessarily true in the general setting.  It follows as a
consequence of the classification of klt surface singularities, however, that
there exists a big set $X° ⊆ X$ where $X°$ has only quotient singularities,
\cite[Proposition~9.3]{GKKP11}.  The full proof of
Theorem~\ref{thm:MYinequality} uses $X°$ in lieu of $X_{\reg}$.  This leads to
fundamental complications.  Following Mumford's seminal paper \cite{MR717614},
the discussion of orbifold Chern classes forces us discuss $X°$ as a
$ℚ$-variety, and to consider global covers of big open subsets of $X°$, which
can be chosen to be Cohen-Macaulay, but not necessarily to be smooth.  We need
to show that all our notions, Higgs sheaves in particular, behave well under the
elementary operation of Mumford's program; ditto for some of Simpson's
constructions and result.  In particular, we need to show that Higgs sheaves can
be pulled into the $ℚ$-variety structure, and from there to any Cohen-Macaulay
cover, and any resolution thereof.  This setting also forces us to develop our
whole theory in the equivariant setting, for varieties with actions of the
appropriate Galois groups.  The failure of reflexive pull-back to have any
functorial properties is a main obstacle there.  For details the reader is
referred to \cite{GKPT15}.


%
%
\svnid{$Id: 09-uniform.tex 116 2016-08-18 12:12:37Z kebekus $}

\section{Characterisation of singular ball quotients}
\label{sec:bq}

\subsection{Smoothness criterion}
\label{ssec:scpc3}
\approvals{Behrouz & yes \\Daniel & yes \\Stefan & yes}

The following smoothness criterion is the centrepiece in our proof of the
uniformisation result, Theorem~\ref{thm:BQ}.  Before returning to the proof of
Theorem~\ref{thm:BQ} in Section~\ref{ssec:prtmss} below, we will therefore
discuss its proof in some detail.

\begin{prop}[\protect{Smoothness criterion, \cite[Proposition~9.3]{GKPT15}}]\label{prop:c3}
  Let $Y$ be a projective variety of dimension $n$ that is smooth in codimension
  two and has at worst canonical singularities.  Assume furthermore that the
  étale fundamental groups of $Y$ and of its smooth locus agree,
  $\what{π}_1(Y_{\reg}) \cong \what{π}_1(Y)$.  If $K_Y$ is ample and if equality
  holds in the Miyaoka-Yau inequality~\eqref{eq:X2}, then $Y$ is smooth.
\end{prop}

Here are the main steps of the proof, which is taken almost verbatim from
\cite[Section~9.2]{GKPT15}.  The main object of study is the canonical Higgs
sheaf $(\sE_Y,θ_Y)$ on $Y$.  In Step~1 we consider this system.  In analogy to
Section~\ref{ssec:toriq}, we fix a complete intersection surface $S$ that
verifies various properties required in the next steps.  This includes
satisfying the property that $(\sE_Y,θ_Y)|_S$ is stable and that a Lefschetz
hyperplane theorem holds.  In Step~2 we construct a \pCVHS{} on $S$ out of this
data, whose induced Higgs bundle is $\sEnd(\sE_Y)|_S$.  It goes without saying
that Simpson's result on the existence of Hermitian-Yang-Mills metrics for
stable Higgs bundles is the key ingredient here.  In Step~3 we extend this
\pCVHS{} to a Higgs bundle $(\sF_{\wtilde Y},θ_{\wtilde Y})$ on a resolution
$\wtilde Y$ of $Y$ and consider local period maps
$$
ρ: \bigl\{\text{1-connected subset of $\wtilde Y$}\bigr\} → \{\text{period domain}\}
$$
Thanks to a factorisation via the period domain,
Proposition~\ref{prop:higgsfromdownst}, we know that $\sF_{\wtilde Y}$ comes
from a locally free sheaf $\sF_Y$ on $Y$.  In the final step we prove that
$\sEnd(\sE_Y) \cong \sF_Y$.  It follows that $\sEnd(\sE_Y)$ is locally free and
then so is $\sT_Y$.  Proposition~\ref{prop:c3} thus follows from the
Lipman-Zariski conjecture for varieties with canonical singularities,
Theorem~\ref{thm:LZ:1}.

We will now go through the steps in more detail.  We aim to present the proof in
a way such that the parallels to Section~\ref{ssec:toriq} become obvious.

\subsubsection*{Step~1: Setup}
\approvals{Behrouz & yes \\Daniel & yes \\Stefan & yes}

We begin by considering the natural Higgs sheaf $(\sE_Y, θ_Y)$, as given in
Example~\ref{ex:BQfield1}, where $\sE_Y = Ω^{[1]}_Y ⊕ \sO_Y$ and
$θ(a + b) = (0+1)⊗a$.  By Lemma~\ref{lem:1} the Higgs sheaf $(\sE_Y, θ_Y)$ is
stable with respect to the ample bundle $K_Y$.

Choose a strong log resolution of singularities, $π : \wtilde Y → Y$, such that
there exists a $π$-ample Cartier divisor supported on the exceptional locus of
$π$.

\begin{claim}\label{claim:flatbounded}
  Write $r := (n+1)²$.  Let ${\sf B}_r$ denote the set of locally free sheaves
  $\sF$ on $X$ that have rank $r$, satisfy
  $μ^{\max}_{K_Y}(\sF) = μ^{\max}_{K_Y}(\sEnd \sE_Y)$, and have Chern classes
  $c_i\bigl(π^* \sF\bigr) = 0$ for all $0 < i ≤ r$.  Then, ${\sf B}_r$ is
  bounded.
\end{claim}
\begin{proof}[Proof of Claim~\ref{claim:flatbounded}]
  Since $X$ has rational singularities, the Euler characteristics $χ_X(\sG)$ and
  $χ_{\wtilde Y}(π^* \sG)$ agree for all locally free sheaves $\sG$ on $Y$.  The
  assumption on Chern classes thus guarantees that the Hilbert polynomials of
  the members $\sF ∈ {\sf B}_r$ are constant, cf.\
  \cite[Corollary~15.2.1]{Fulton98}.  Boundedness thus follows from
  \cite[Theorem~3.3.7]{MR2665168}.  This ends the proof of
  Claim~\ref{claim:flatbounded}.
\end{proof}

Next, choose a sufficiently large and divisible integer $m \gg 0$, and choose a
general tuple of hyperplanes $H_1, …, H_{n-2} ∈ |m·K_X|$, with general complete
intersection surface $S := H_1 ∩ \cdots ∩ H_{n-2}$.  The following items will
then hold.

\begin{enumerate}
\item\label{il:whitehorse} The intersection $S$ is a smooth surface, and
  entirely contained in $Y_{\reg}$.  This is because $Y$ is smooth in
  codimension two by assumption.

\item\label{il:redhorse} The restriction $(\sE_Y,θ_Y)|_S$ is stable with respect
  to $K_Y|_S$.  This follows from the Restriction Theorem~\ref{thm:restriction}.

\item\label{il:blackhorse} The natural morphism $ι_*: π_1(S) → π_1(Y_{\reg})$,
  induced by the inclusion $ι: S \into Y_{\reg}$, is isomorphic.  This is the
  content of Goresky-MacPherson's Lefschetz hyperplane theorem for homotopy
  groups, \cite[Theorem in Section~II.1.2]{GoreskyMacPherson}.

\item\label{il:dappledhorse} Let $\sF ∈ {\sf B}_r$.  Then, $\sF$ is isomorphic
  to $\sEnd \sE_Y$ if and only if the restrictions $\sF|_S$ and
  $(\sEnd \sE_Y)|_S$ are isomorphic.  This is a consequence of the boundedness
  statement in Claim~\ref{claim:flatbounded}, and of a Bertini-type theorem for
  isomorphism classes in bounded families \cite[Corollary~5.3]{GKP13}.
\end{enumerate}

\begin{claim}\label{claim:blackhorse}
  The natural morphism $π_1(S) → π_1(Y)$ is surjective and induces an
  isomorphism of profinite completions.
\end{claim}
\begin{proof}
  The natural morphism $π_1(Y_{\reg}) → π_1(Y)$ is surjective, \cite[0.7.B on
  Page~33]{FL81}, and induces an isomorphism of profinite completions by
  assumption.  Composed with the inclusion $S \hookrightarrow Y_{\reg}$,
  Claim~\ref{claim:blackhorse} follows from Item~\ref{il:blackhorse} above.
\end{proof}

\subsubsection*{Step~2: Construction of a \pCVHS\ on $S$}
\approvals{Behrouz & yes \\Daniel & yes \\Stefan & yes}

Since $S$ is entirely contained in the smooth locus of $Y$, the restricted Higgs
sheaf $(\sE_Y,θ_Y)|_S$ is actually a Higgs bundle, and
Construction~\ref{cons:dl} allows to equip the corresponding endomorphism bundle
with a Higgs field.  For brevity of notation, set
$\sF_S := \sEnd \bigl( \sE_Y \bigr)|_S$ and write $(\sF_S, Θ_S)$ for associated
Higgs bundle, constructed as in \ref{cons:dl}.  The rank of $\sF_S$ equals
$r = (n+1)²$.
  
\begin{claim}\label{claim:HVHS}
  The Higgs bundle $(\sF_S, Θ_S)$ is induced by a \pCVHS, in the sense of
  Definition~\ref{def:hscvhs}.
\end{claim}
\begin{proof}[Proof of Claim~\ref{claim:HVHS}]
  We need to check the properties listed in Theorem~\ref{thm:charPCVHS}.

  \smallskip

  Item~\ref{il:emerson}: polystability with respect to $K_Y|_S$.  By
  Theorem~\ref{thm:restriction}, we know that both $(\sE_Y,θ_Y)|_S$ and its dual
  are $K_Y|_S$-stable Higgs bundles on the smooth surface $S$.  In particular,
  it follows from~\cite[Theorem~1(2)]{MR1179076} that both bundles carry a
  Hermitian-Yang-Mills metric with respect to $K_X|_S$, and thus so does
  $(\sF_S, Θ_S)$.  Hence it follows from~\cite[Theorem~1]{MR1179076} that
  $(\sF_S, Θ_S)$ is polystable with respect to $K_Y|_S$.

  \smallskip

  Item~\ref{il:lake}: vanishing of Chern classes.  As $\sF_S$ is the
  endomorphism bundle of the locally free sheaf $\sE_Y|_S$, its first Chern
  class vanishes.  Vanishing of $c_2(\sF_S)$ is then an immediate consequence of
  the assumed equality in \eqref{eq:X2}.  Together with polystability, this
  implies that $\sF_S$ is flat, \cite[Theorem~1]{MR1179076}, and hence all its
  Chern classes vanish.

  \smallskip

  Item~\ref{il:palmer}: we have seen in Example~\ref{ex:BQfield0} that $\sE_Y$
  has the structure of a system of Hodge bundles.  Its isomorphism class is
  therefore fixed under the action of $ℂ^*$, \cite[Page~45]{MR1179076}.
  Observing that the same holds for its dual and its endomorphism bundle, this
  ends the proof of Claim~\ref{claim:HVHS}.
\end{proof}

\subsubsection*{Step~3: Extension of the \pCVHS{} to a resolution}
\approvals{Behrouz & yes \\Daniel & yes \\Stefan & yes}

Since $S$ is entirely contained in the smooth locus of $Y$, it is canonically
isomorphic to its preimage $\wtilde{S} := π^{-1}(S)$ in the resolution
$\wtilde X$.  Let $(\sF_{\wtilde{S}}, Θ_{\wtilde{S}})$ be the Higgs bundle on
$\wtilde{S}$ that corresponds to $(\sF_S, Θ_S)$ under this isomorphism.

There exists a $ℚ$-divisor $E ∈ ℚ\Div(\wtilde Y)$, supported entirely on the
$π$-exceptional locus, such that $\wtilde H := π^*(K_Y)+E$ is ample.  Since
$\wtilde S$ and $\supp E$ are disjoint, the Higgs bundle
$(\sF_{\wtilde{S}}, Θ_{\wtilde{S}})$ is clearly semistable with respect to
$\wtilde H|_{\wtilde{S}}$.

Recall from \cite[Theorem~1.1]{Takayama2003} that the natural map of fundamental
groups, $π_1(\wtilde Y) → π_1(Y)$ is isomorphic.  Together with
Claim~\ref{claim:blackhorse}, this implies that
$π_1(\wtilde S) → π_1(\wtilde Y)$ is surjective, and induces an isomorphism of
profinite completions.  Items~\ref{il:garfunkel} and \ref{il:simon} of
Corollary~\ref{cor:pfccs} therefore allow to find a Higgs bundle
$(\sF_{\wtilde Y}, Θ_{\wtilde Y})$ on $\wtilde Y$ that restricts to
$(\sF_{\wtilde S}, Θ_{\wtilde{S}})$, and is induced by \pCVHS.  We have seen in
Remark~\ref{rem:charPCVHS} that all Chern classes of $\sF_{\wtilde Y}$ vanish.

\subsubsection*{Step~4: Identification of the \pCVHS}
\approvals{Behrouz & yes \\Daniel & yes \\Stefan & yes}

We have seen in Proposition~\ref{prop:higgsfromdownst} that $\sF_{\wtilde Y}$
comes from $Y$.  More precisely, there exists a locally free sheaf $\sF_Y$ on
$Y$ such that $\sF_{\wtilde Y} =π^*(\sF_Y)$.  First notice that $\sF_Y$ is a
member of the family ${\sf B}_r$ that was introduced in
Claim~\vref{claim:flatbounded}.  Item~\ref{il:dappledhorse} thus gives an
isomorphism $\sEnd \sE_Y \cong \sF_Y$, showing that $\sEnd \sE_Y$ is locally
free.  But $\sEnd \sE_Y$ contains $\sT_Y$ as a direct summand.  It follows that
$\sT_Y$ is locally-free and thus $Y$ is smooth by the partial solution of the
Zariski-Lipman problem for spaces with canonical singularities,
Theorem~\ref{thm:LZ:1}.  This finishes the proof of the smoothness criterion,
Proposition~\ref{prop:c3}.  \qed

\subsection{Proof of Theorem~\ref*{thm:BQ} in a simplified setting}
\label{ssec:prtmss}
\approvals{Behrouz & yes \\Daniel & yes \\Stefan & yes}

For the purposes of this survey, we prove Theorem~\ref{thm:BQ} only under the
following simplifying assumptions.  Section~\ref{ssec:pigc2} briefly discusses
the missing pieces for a proof in the general case.

\begin{assumption}\label{ass:simpl2}
  The canonical bundle $K_X$ is ample, and $X$ is therefore equal to its
  canonical model.
\end{assumption}

Recalling from Definition~\ref{def:minimal} that minimal varieties have terminal
singularities, we infer that $X$ is smooth in codimension two.  In particular,
Chern classes $c_1$ and $c_2$ exist.

Now consider a maximally quasi-étale cover $f : Y → X$, as given by
Theorem~\ref{thm:gfa:2}.  Since $f$ is unramified in codimension two we find
that $K_Y = f^*(K_X)$ is also ample and that $Y$ again has terminal
singularities, cf.~\cite[Proposition~1.2.13]{Laz04-I} and
\cite[Proposition~5.20]{KM98}.  Since $\sT_Y$ and $f^* \sT_X$ differ only along
a set of codimension three, the projection formula for Chern classes yields that
\begin{equation}
  \Bigl( 2(n+1)· c_2(\sT_Y)-n·c_1(\sT_Y)^2
  \Bigr)·[K_Y]^{n-2}= 0.
\end{equation}
In other words, equality holds in the Miyaoka-Yau inequality for $Y$.  In
particular, the smoothness criterion of Proposition~\ref{prop:c3} applies,
showing that $Y$ is smooth.  So, $Y$ is uniformised by the ball, thanks to the
original result of Yau, Theorem~\ref{thm:YausTheorem}.  This finishes the proof
of Theorem~\ref{thm:BQ} in the simplified setting of
Assumption~\ref{ass:simpl2}.  \qed

\subsection{Proof in the general case}
\label{ssec:pigc2}
\approvals{Behrouz & yes \\Daniel & yes \\Stefan & yes}

To prove Theorem~\ref{thm:BQ} in general, we show that the tangent sheaf of the
canonical model satisfies the equality in Miyaoka-Yau inequality, and that it is
smooth in codimension two.  This is a consequence of two computations with
orbifold Chern classes:

Let $π: X → X_{can}$ be the morphism from $X$ to its canonical
model.

\subsubsection*{The Miyaoka-Yau equality for $\sT_{X_{can}}$}
\approvals{Behrouz & yes \\Daniel & yes \\Stefan & yes}
\CounterStep

We claim that $\sT_{X_{can}}$ verifies the Miyaoka-Yau equality.  Reason: on the
one hand we know from Theorem~\ref{thm:MYinequality} that $\sT_{X_{can}}$
verifies the Miyaoka-Yau inequality.  On the other hand, Chern classes
calculations similar to \cite[Proposition~1.1]{SBW94} show that
\begin{equation}\label{eq:c2up}
  c_2\bigl(\sT_X\bigr)·[K_X]^{n-2}- \what{c}_2\bigl(\sT_{X_{can}}\bigr)·[K_{X_{can}}]^{n-2}
  = c_2\bigl(\sT_{\wtilde S}\bigr) - \widehat c_2\bigl(\sT_S\bigr) ≥ 0,
\end{equation}
where $\wtilde S$ is the birational transform of a complete intersection surface
$S = D_1 ∩ \cdots ∩ D_{n-2}$, for sufficiently general members $D_i$ of
$|m· K_{X_{can}}|$, for $m$ sufficiently large and divisible.  But this implies
that the MY discriminant for $\sT_{X_{can}}$ is bounded from the above by the
one for $\sT_X$.  In other words, the MY discriminant of $\sT_{X_{can}}$ is at
most zero.

\subsubsection*{The singularities of the canonical model}
\approvals{Behrouz & yes \\Daniel & yes \\Stefan & yes}

As the MY discriminant of $\sT_{X_{can}}$ is equal to zero, \eqref{eq:c2up}
implies that
\begin{equation}\label{eq:c2}
  c_2(\sT_{\wtilde S})=\widehat c_2(\sT_S).
\end{equation}
But as $X_{can}$ has only canonical singularities, every connected exceptional
divisor of $π|_{\wtilde S} : \wtilde S → S$ is a tree of $ℙ^1$s.  The fact that
$\what{c}_2$ computes the orbifold Euler characteristic of $S$ implies that the
equality can only hold if $S$ is smooth.  But if general complete intersections
surfaces are smooth, then $X_{can}$ needs to be smooth in codimension two.


\vspace{1cm}

\end{document}